\documentclass[12pt,letterpaper]{amsart}
\usepackage[utf8]{inputenc}
\usepackage{graphicx}
\usepackage[all]{xypic}
\theoremstyle{definition}
\newtheorem{df}{Definition}
\theoremstyle{plain}
\newtheorem{stat}{Statement}

\newtheorem{thm}{Theorem}
\newtheorem{lm}{Lemma}
\newtheorem{cor}{Corollary}
\theoremstyle{remark}
\newtheorem{ex}{Example}
\newtheorem{rem}{Remark}
\author{F. A. Arias and M. Malakhaltsev}
\address{Universidad de los Andes, Bogotá, Colombia}
\email{fa.arias44@uniandes.edu.co; mikarm@uniandes.edu.co}

\title{A generalization of the Gauss-Bonnet and Hopf-Poincaré theorems}

\begin{document}
\maketitle

\begin{abstract}
We consider a locally trivial fiber bundle $\pi : E \to M$ over a compact oriented two-dimensional manifold $M$, and a section $s$ of this bundle defined over $M \setminus \Sigma$, where $\Sigma$ is a discrete subset of $M$.
We call the set $\Sigma$ the \emph{set of singularities of the section} $s : M \setminus \Sigma \to E$.

We assume that the behavior of the section $s$ at the singularities is controlled in the following way:
$s(M \setminus \Sigma)$ coincides with the interior part of a surface $S \subset E$ with boundary $\partial S$, and $\partial S$ is $\pi^{-1}(\Sigma)$. 
For such sections $s$ we define an index of $s$ at a point of $\Sigma$, which generalizes in the natural way the index of zero of a vector field, and then prove that the sum of this indices at the points of $\Sigma$ can be expressed as integral over $S$ of a $2$-form constructed via a connection in $E$.

Then we show that the classical Hopf-Poincaré-Gauss-Bonnet formula is a partial case of our result, and consider some other applications.
\end{abstract}



\section{Introduction}
\label{sec:1}

Let $M$ be a $2$-dimensional compact oriented manifold, and $V$ be a vector field on $M$.
The classical Hopf-Poincaré theorem affirms that the sum of indices of zeros of $V$ is equal to the Euler characteristic of $M$:
\begin{equation}
\sum_{z \in Z(V)} ind_{z}(V) = \chi(M),
\label{eq:1}
\end{equation}
where $Z(V)$ is the set of zeros of $V$ and $ind_{z}(V)$ is the index of $V$ at $z$ (see for example~\cite{bott_tu1982}).
If, in addition, the manifold $M$ is endowed by a Riemannian metric $g$, then the Gauss-Bonnet theorem says that the Euler characteristic is expressed in terms of the curvature $K$ of this metric:
\begin{equation}
\chi(M) = \frac{1}{2\pi} \int_M K d\sigma,
\label{eq:2}
\end{equation}
where $d\sigma$ is the area form of $g$.
Therefore, we have that
\begin{equation}
\sum_{z \in Z(V)} ind_{z}(V) = \frac{1}{2\pi} \int_M K d\sigma,
\label{eq:3}
\end{equation}
Then this formula can be generalized for sections of a unit sphere subbundle of the tangent bundle of a Riemannian manifold, thus we arrive at the Euler class (see, e.\, g.~\cite{kobayashi_nomizu1969}, Appendix 20: Gauss-Bonnet theorem) and it also can be generalized to arbitrary sphere bundles~\cite{bott_tu1982}.

In this paper we find a generalization of this formula to arbitrary locally trivial bundles with compact fibers over two-dimensional closed manifolds.
Namely, we consider a locally trivial fiber bundle $\pi : E \to M$ over a compact oriented two-dimensional manifold $M$, and a section $s$ of this bundle defined over $M \setminus \Sigma$, where $\Sigma$ is a discrete subset of $M$.
We call the set $\Sigma$ the \emph{set of singularities of the section} $s : M \setminus \Sigma \to E$.

We assume that the behavior of the section $s$ at the singularities is controlled in the following way:
$s(M \setminus \Sigma)$ coincides with the interior part of a surface $S \subset E$ with boundary $\partial S$, and $\partial S$ is $\pi^{-1}(\Sigma)$ (see details in subsection~\ref{subsec:2_1}). 
It is worth noting that this idea was used by S.-S.Chern in~\cite{chern_1944}.  
For such sections $s$ we define an index of $s$ at a point of $\Sigma$, which generalizes in the natural way the index of zero of a vector field (subsection~\ref{subsec:2_2}), and then prove that the sum of these indices at the points of $\Sigma$ can be expressed as integral over $S$ of a $2$-form constructed via a connection in $E$.
Thus in Theorem~\ref{thm:2} we obtain a generalization of the formula~\eqref{eq:3}.

In Section~\ref{sec:4} we show that the classical fact that the index of a vector field on a closed two-dimensional Riemann manifold is the integral of the curvature of the Levi-Civita connection can be obtained from our generalized version.
Also in this section we consider the projective bundles and projective connections and find a relation between the index of a section of a projective bundle and the curvature of a projective connection in this bundle.

In Section~\ref{sec:5} our results are applied to the theory of $G$-structures with singularities.
If $\bar P \to M$ is a $\bar G$-principal bundle, which is reduced to a $G$-principal bundle over $M \setminus \Sigma$, where $\Sigma$ is a submanifold of $M$, we say that $\bar P$ is a $G$-principal bundle with singular set $\Sigma$.
If the $\bar G$-principal bundle $\bar P$ is a subbundle of the linear frame bundle $L(M)$ of a manifold $M$, we say that $\bar P$ is a $G$-structure with singularities.
Almost all singularities which appear in classical differential geometry define $G$-structures with singularities.
Some examples are: a vector field on a Riemannian $n$-dimensional manifold defines $SO(n-1)$-structure with singularities, a metric on an $n$-dimensional manifold $M$ which degenerates along a submanifold, determines a $SO(n)$-structure with singularities on $M$ (examples of such metrics were considered for example in~\cite{arteaga_malakhaltsev_trejos_2012}), $3$-webs with singularities on a two-dimensional manifolds, which in particular appear in considerations of algebraic ordinary differential equations on manifolds (see, e.\,g.,~\cite{agafonov2009}) define a $G$-structure with singularities, where $G \subset GL(2)$ is the Lie subgroup of diagonal matrices~\cite{arias_arteaga_malakhaltsev2015}, also sub-Riemannian manifolds with singularities determine the corresponding $G$-structure with singularities~\cite{arteaga_malakhaltsev2011}.

A $G$-principal bundle $\bar P$ with singularities determines a section $s : M \setminus \Sigma \to \bar P/G$, which is a section with singularities of the bundle $\bar P/G$.
For the case $\dim M = 2$, we apply the results obtained in Section~\ref{sec:2} to this situation and obtain the relation between the index of this section and the connection in the bundle $\bar P \to M$ (see Section~\ref{sec:5} for details).

\section{Singularity of section and its index}
\label{sec:2}

\subsection{Sections with singularities. Resolution of singularity}
\label{subsec:2_1}

Let $M$ be a $2$-dimensional oriented closed manifold, and $\xi = (\pi : E \to M)$ a locally trivial fiber bundle with typical fiber $F$ and structure group $G$.
We will assume that $G$ is a connected Lie group.
For each $x \in M$, by $F_x$ we will denote the fiber of $\xi$ over $x$, and by $i_x : F_x \to E$ the inclusion.
Let $\Sigma \subset M$ be a discrete subset.
Let us consider a section $s : M \setminus \Sigma \to E$, we will call $\Sigma$ the \emph{singular set} of the section $s$.

Assume that \emph{there exist a $2$-dimensional oriented compact manifold $S$ with boundary $\partial S$ and a map $q : S \to E$ such that
the map $q$ restricted to $\overset{\circ}{S} = S \setminus \partial S$ is a diffeomorphism between $S \setminus \partial S$ and $s(M \setminus \Sigma)$}.
We will call the map $q$ a \emph{resolution of the singularities of the section $s$}.

\begin{thm}
\emph{a)} $\pi \circ q : S \setminus \partial S \to M \setminus \Sigma$ is a diffeomorphism.

\emph{b)} The map $\pi \circ q : S \to M$ is surjective.

\emph{c)} $\pi \circ q(\partial S) = \Sigma$.

\emph{d)} Let $\partial S = \bigcup\limits_{i = \overline{1, k}} \partial S_i$, where $\partial S_i$ are the boundary components.
Then for any $x \in \Sigma$, there exist one and only one boundary component $\partial S_i$ such that $\pi^{-1}(x) = q(\partial S_i)$.

\emph{e)} For any $x \in \Sigma$ exists a neighborhood $U$ of $x$ such that
\begin{enumerate}
\item
$U$ is diffeomorphic to a disk;
\item
There exists a closed set $V \subset S$ diffeomorphic to a ring such that the boundary of $V$ consists of two connected components $C$ and $C'$, where the connected component $C$ is a closed curve in $S \setminus \partial S$ and the other connected component $C'$ is the connected component $\partial S_i$ of the boundary $\partial S$ which corresponds to $x$ by d), and
\begin{equation}
\pi \circ q(C) = \partial U, \quad \pi \circ q(V \setminus \partial V) = U \setminus x, \quad \pi \circ q(\partial S_i) = x
\label{eq:4}
\end{equation}
\end{enumerate}

\label{thm:1}
\end{thm}
\begin{rem}
\emph{In what follows we will label each point $x \in \Sigma$ such that $\pi^{-1}(x) = q(\partial S_i)$ by the corresponding index $i$}.
\label{rem:1}
\end{rem}

\begin{proof}
a) $s : M \setminus \Sigma \to s(M \setminus \Sigma)$ is a diffeomorphism and $\left. \pi \right|_{s(M \setminus \Sigma)} : s(M \setminus \Sigma) \to M \setminus \Sigma$ is the diffeomorphism inverse to $s$.
Therefore $\pi \circ q : S \setminus \partial S \to M \setminus \Sigma$ is a diffeomorphism

b) For each point $x \in M \setminus \Sigma$, the point $y = q^{-1}(s(x)) \in S$ has the property that $\pi \circ q(y) = x$.
Let us take a point $x \in \Sigma$ and a sequence of points $\left\{ x_n \in M \setminus \Sigma \right\}$ converging to $x$.
The sequence of points $y_n = q^{-1}(s(x_n)) \in S$ has a convergent subsequence $\left\{ y_{n_k} \right\}$ because $S$ is compact.
If $y = \lim_{n_k \to \infty} y_{n_k}$, then
\begin{equation}
\pi \circ q(y) = \lim_{n_k \to \infty} \pi q(y_{n_k}) = \lim_{n_k \to \infty} x_{n_k} = x.
\label{eq:5}
\end{equation}

c) As $\pi \circ q$ is surjective and $\pi \circ q(S \setminus \partial S) = M \setminus \Sigma$ because $q(S \setminus \partial S) = s(M \setminus \Sigma)$, we have that $\pi \circ q(\partial S) = \Sigma$.

d) and e)
For a point $x \in \Sigma$ let us take a neighborhood $U$ diffeomorphic to an open disk such that $U_0 = U \setminus \left\{ x \right\} \subset M \setminus \Sigma$.
The closure $\overline{U_0} = \Gamma \sqcup U_0 \sqcup \left\{ x \right\}$, where $\Gamma$ is the boundary of $U$.

Let $V_0 = q^{-1} s(U_0)$. Then $\pi \circ q(\overline{V_0}) = \overline{U_0}$ because $M$ and $S$ are compact.
As $\pi \circ q : S \setminus \partial S \to M \setminus \Sigma$ is a diffeomorphism, and $\Gamma \subset M \setminus \Sigma$, we have a curve $C \subset (\overline{V_0} \setminus V_0) \cap (S \setminus \partial S)$ such that $\pi \circ q(C) = \Gamma$.
This means that $C$ is a part of the boundary $\partial V_0$ which lies in the interior part of $S$.
Now, by b) we know that $\pi^{-1}(x) \subset \sqcup_{i \in I} \partial S_i$, $I \subset \{1, \cdots, n\}$.
Finally we have that $\overline{V_0} = C \sqcup V_0 \sqcup \left( \sqcup_{i \in I} \partial S_i \right)$, and we know that $V_0$ is diffeomorphic to $U_0$, and $C \subset S \setminus \partial S$ is diffeomorphic to $\Gamma$, that is $C$ is diffeomorphic to a circle $\mathbb{S}^1$.
Thus $V_0$ is an open set in $S$ diffeomorphic to an open ring, and the boundary of $V_0$ consists of a closed curve in the interior part of $S$ and some components of the boundary of $S$.
However, this means that the boundary of $V_0$ contains one and only one component $\partial S_i$ of $\partial S$, otherwise the fundamental group of $\pi_1(V_0) \not\cong \mathbb{Z}$.
Thus we get the affirmation d) and the decomposition $V = \overline{V_0} = C \sqcup V_0 \sqcup \partial S_i$ with the properties given by~\eqref{eq:4}.
Therefore, we have proved e).
\end{proof}
\begin{rem}
From the proof of items d) and e) of Theorem~\ref{thm:1} it follows that, for each boundary component $\partial S_i$, we have a closed neighborhood $V_i$ of $\partial S_i$ and a diffeomorphism of manifolds with boundary $u_i : \mathbb{S}^1 \times [0, 1] \to V_i$ such that $u_i(\mathbb{S}^1 \times \left\{ 1 \right\} = \partial S_i$, $U = \pi \circ q \circ u_i(\mathbb{S}^1 \times [0, 1])$ is a closed neighborhood of $x_i \in \Sigma$ diffeomorphic to a disk, $\pi \circ q \circ u_i(\mathbb{S}^1 \times \left\{ 0 \right\}$ is the boundary of $U$, and $\pi \circ q \circ u_i(\mathbb{S}^1 \times \left\{ 1 \right\} = x$ (see Figure~\ref{fig:1}).
\begin{figure}[h]
\centering
\includegraphics[scale=0.1]{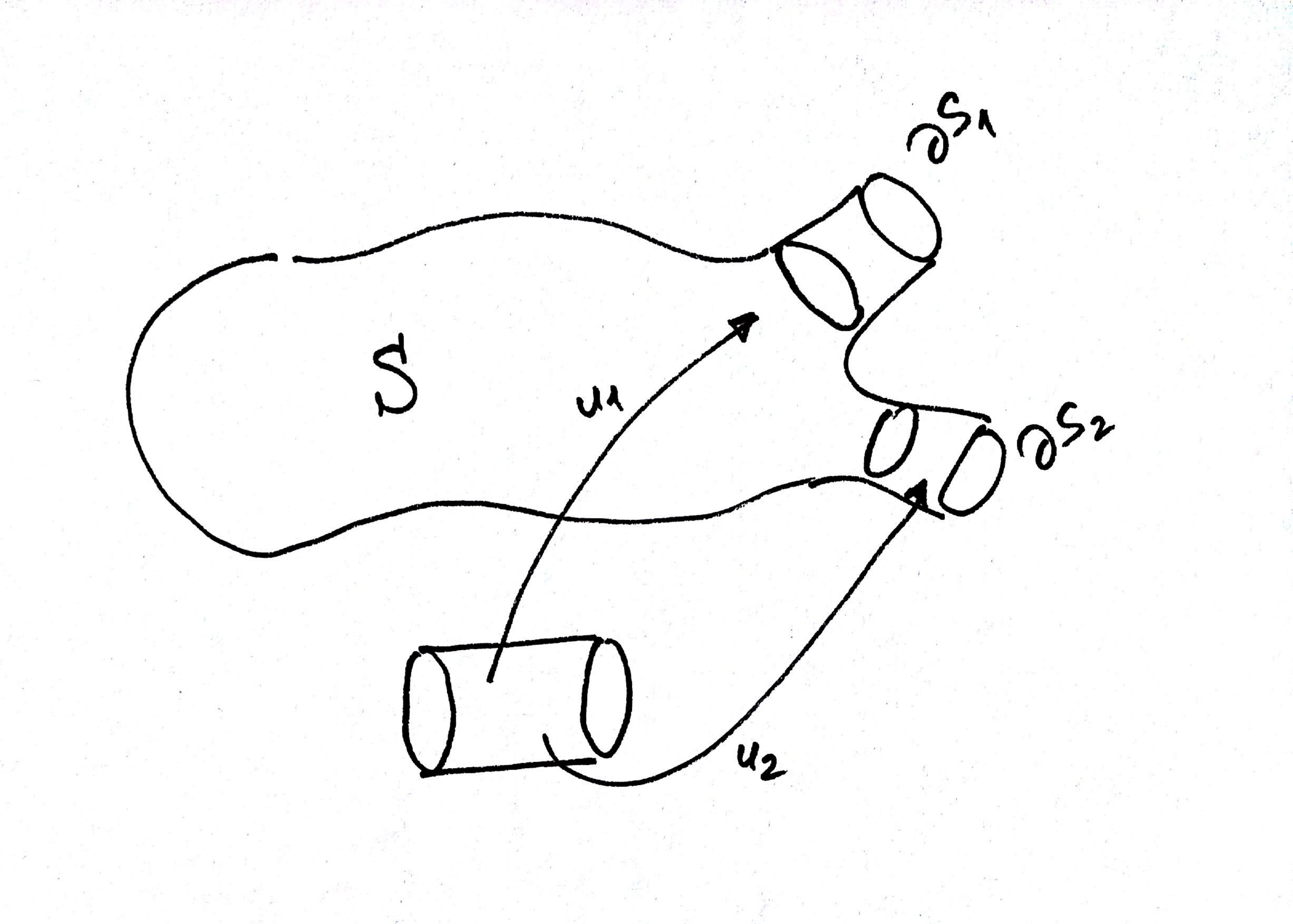}
\caption{The surface $S$}
\label{fig:1}
\end{figure}
\label{rem:2}
\end{rem}

\begin{ex}
Let $\xi$ be a unit vector bundle of $\mathbb{R}^2$, $\pi : \mathbb{R}^2 \times \mathbb{S}^1 \to \mathbb{R}^2$, and consider a section $s : \mathbb{R}^2 \to E$ which takes value $s(\rho, \varphi) = (\cos n\varphi, \sin n\varphi)$ with respect to the polar coordinate system on $\mathbb{R}^2 \setminus (0, 0)$.
The point $(0, 0)$ is the singular point of this section.
If we take the coordinate system $(\rho, \varphi, \psi)$ on $\pi^{-1}(\mathbb{R}^2 \setminus (0, 0))$, then $s(\rho, \varphi) = (\rho, \varphi, n\varphi)$.
Now the map $u : \mathbb{S}^1 \times [0, 1] \to E$ is given with respect to coordinates $(x, y, \psi)$ on $E = \mathbb{R}^2 \times \mathbb{S}^1$ by $u(\alpha, t) = ( (1 - t)\cos\alpha, (1-t)\sin\alpha, n \alpha)$.
Therefore $u( \mathbb{S}^1 \times \left\{ 0 \right\})$ is a curve consisting of helices, $u( \mathbb{S}^1 \times \left\{ 1 \right\})$ is the fiber $(0, 0) \times \mathbb{S}^1 \subset E$ represents the image of $u$.
\begin{figure}[h]
\centering
\includegraphics[scale=0.4]{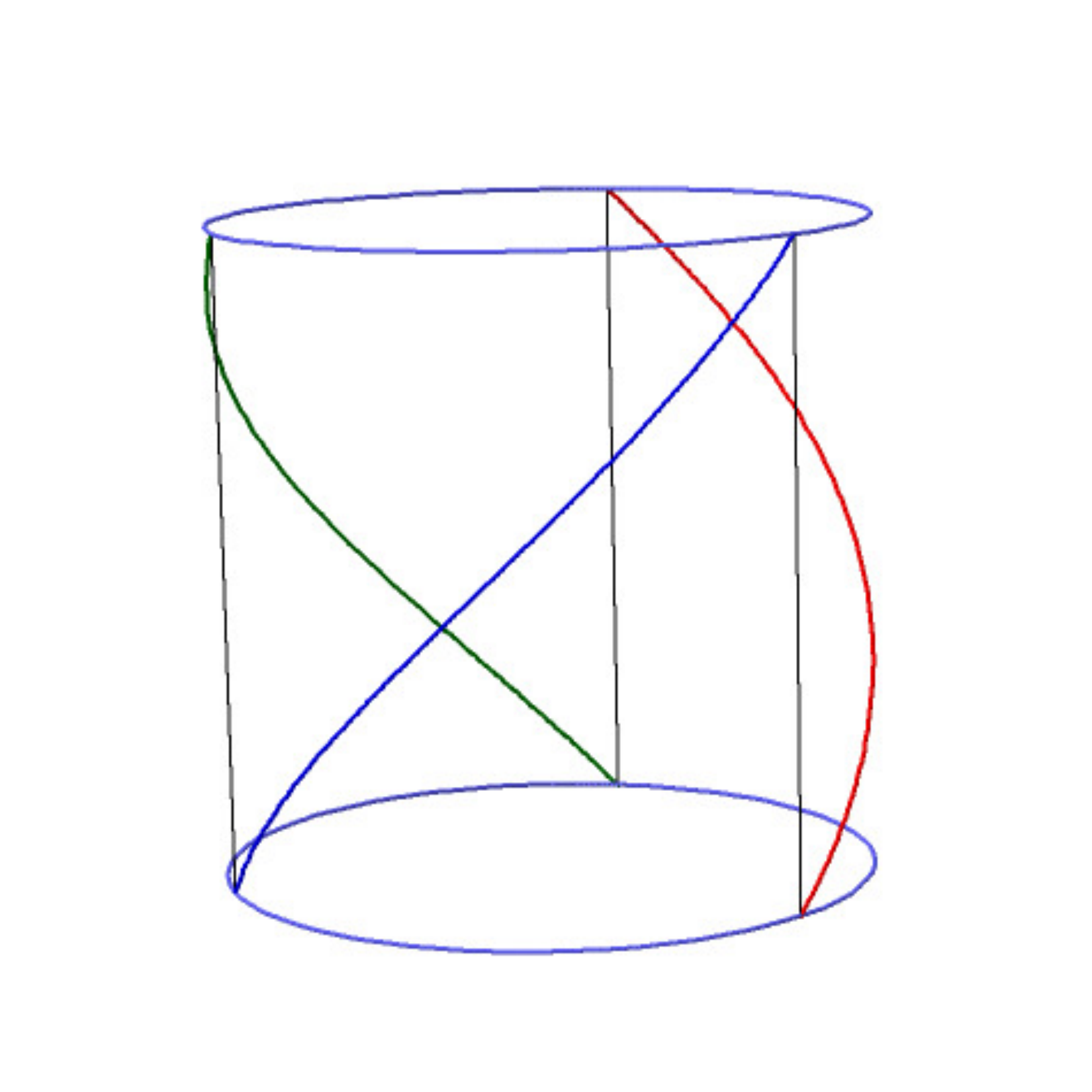}
\caption{Example~\ref{ex:1}}
\label{fig:2}
\end{figure}
\begin{figure}[h]
\centering
\includegraphics[scale=0.4]{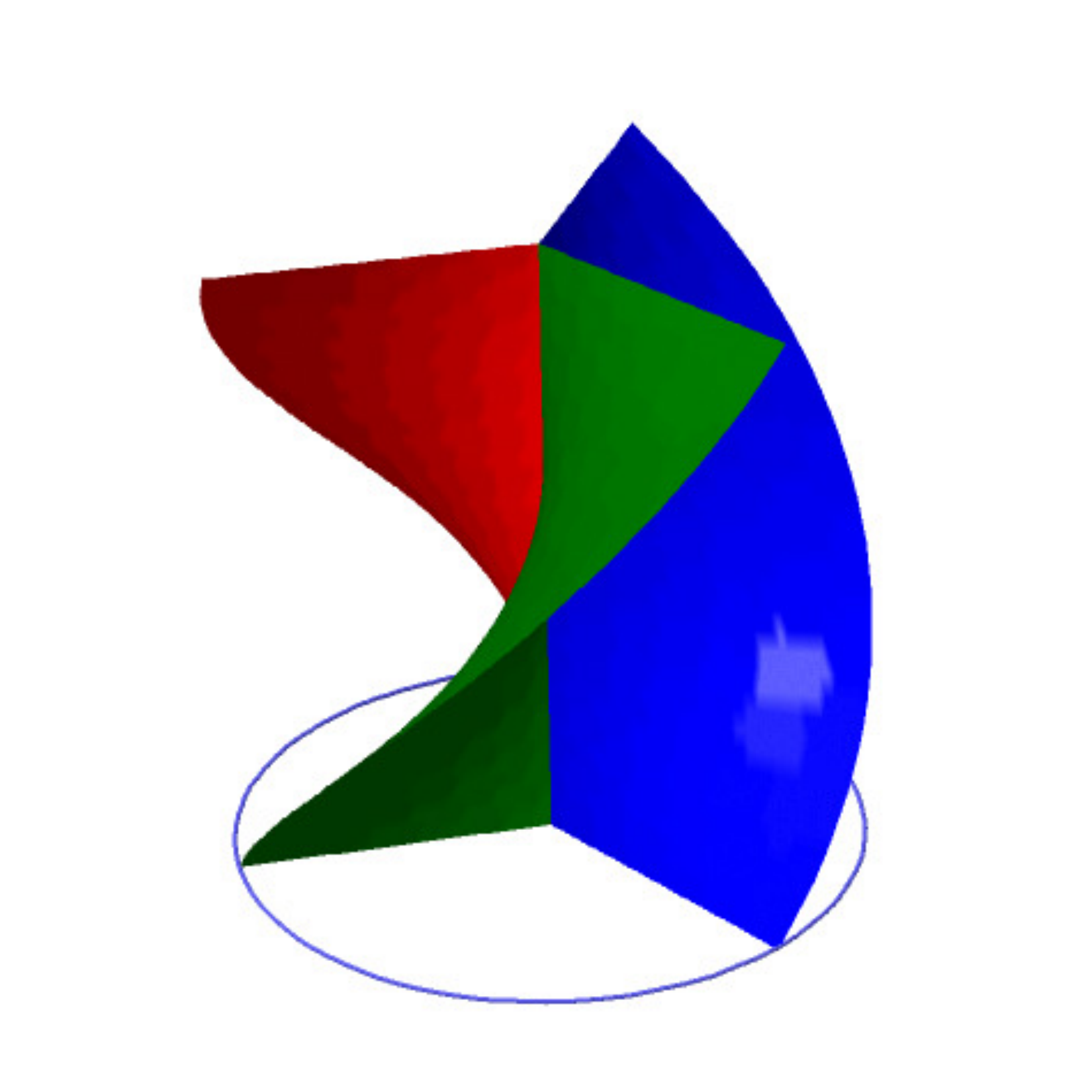}
\caption{Example~\ref{ex:1}}
\label{fig:3}
\end{figure}
\label{ex:1}
\end{ex}

\subsection{Index of a point $x \in \Sigma$.}
\label{subsec:2_2}
First let us first observe that $\pi : E \to M$ is trivial on the homotopy level.
Let $(U, \psi: \pi^{-1}(U) \to U \times F)$ be a chart of the atlas of $\xi$.
Let
\begin{equation}
\eta = p_F \circ \psi: \pi^{-1}(U) \to F,
\label{eq:6}
\end{equation}
where $p_F : U \times F \to F$ is the canonical projection onto $F$.
For each $x \in U$ the map $\eta$ restricted to $F_x = \pi^{-1}(x)$ induces a diffeomorphism $\eta_x : F_x \to F$.
Note that
if we take another chart $(U', \psi': \pi^{-1}(U') \to U' \times F)$, and $\eta' : \pi^{-1}(U') \to F$ is the corresponding map, then on $\pi^{-1}(U \bigcap U')$ we have that
\begin{equation}
\psi' \circ \psi^{-1} : (U \cap U') \times F \to (U \cap U') \times F, \quad (x, y) \mapsto (x, g(x) y),
\label{eq:7}
\end{equation}
where $g : U \cap U' \to G$ is the gluing map of the charts.
Now, for any $x \in U \cap U'$, we have $\eta'_x \circ \eta_x^{-1} (y) = g(x) y$, and, as $G$ is connected, $\eta'_x \circ \eta_x^{-1} : F \to F$ is homotopic to the identity map.
This means that for any $x \in M$ we have well defined isomorphisms of the homotopy and (co)homology groups:
\begin{equation}
\begin{array}{l}
\pi_*(\eta_x) : \pi_*(F_x) \to \pi_*(F),
\\[5pt]
H_*(\eta_x) : H_*(F_x) \to H_*(F), \quad H^*(\eta_x) : H^*(F) \to H^*(F_x),
\end{array}
\label{eq:8}
\end{equation}
which do not depend on the chart.
In particular, the following diagram is commutative:
\begin{equation}
\xymatrix{\pi_1(F_x) \times H^1(F_x) \ar[dr]_{\langle \cdot, \cdot \rangle}\ar[rr]^{\pi_1(\eta_x)\times H^1(\eta_x^{-1})} & & \pi_1(F) \times H^1(F) \ar[dl]^{\langle \cdot, \cdot \rangle}
\\
&\mathbb{R} &
}
\label{eq:9}
\end{equation}
where
$\langle \cdot, \cdot \rangle$ is the natural pairing through the integration.
This means that, for $[\gamma] \in \pi_1(F_x)$ and $[\alpha] \in H^1(F)$, we have
\begin{equation}
\left\langle [\gamma], H^1(\eta_x)[\alpha] \right\rangle = \int_\gamma \eta_x^* \alpha = \int_{\eta_x \gamma} \alpha = \left\langle \pi_1(\eta_x)[\gamma], [\alpha] \right\rangle,
\label{eq:10}
\end{equation}
and this equality neither depends on the chart, nor on the choice of representatives $\gamma$ or $\alpha$ of their equivalence classes.
\begin{df}
For $x_i \in \Sigma$ the map $q \circ u_i : \mathbb{S}^1 \times [0, 1] \to E$ restricted to $\mathbb{S}^1 \times \left\{ 1 \right\}$ gives a map $\gamma_i : \mathbb{S}^1 \to F_{x_i}$, and therefore an element $ind_{x_i}(s) \in \pi_1(F)$ via the isomorphism $\pi_1(\eta_x)$ (see~\eqref{eq:8}).
This element $ind_{x_i}(s)$ will be called the \emph{index of the section $s$ at the point $x_i \in \Sigma$}.
\label{df:1}
\end{df}
\begin{rem}
The index does not depend on the choice of the maps $u_i$.
\label{rem:3}
\end{rem}
\begin{ex}
The index of the section $s$ at the point $(0, 0)$ in Example~\ref{ex:1} is equal to $n \gamma$, where $\gamma$ is a generator of the group $\pi_1(\mathbb{S}^1)$.
\label{ex:2}
\end{ex}
\begin{df}
For $a \in H^1(F)$ we define the \emph{index $ind_{x_i}(s;a)$ of the section $s$ at the point $x_i \in \Sigma$ with respect to $a$} as
\begin{equation}
ind_{x_i}(s;a) = \langle a, ind_{x_i}(s) \rangle = \int_\gamma \alpha = \int_{\mathbb{S}^1} \gamma^* \alpha,
\label{eq:11}
\end{equation}
where $\gamma : \mathbb{S}^1 \to F$ is the curve representing the element $ind_{x_i}(s) \in \pi_1(F)$ and $\alpha \in \Omega^1(F)$ is a closed form representing the element $a \in H^1(F)$.
\emph{The index of a section $s$ with respect to $a$} is the number
\begin{equation}
ind(s; a) = \sum_{x_i \in \Sigma} ind_{x_i}(s;a).
\label{eq:12}
\end{equation}
\label{df:2}
\end{df}
\begin{rem}
By~\eqref{eq:11},
\begin{equation}
ind_{x_i}(s;a) = \int_{\gamma_i} \eta_x^*\alpha,
\label{eq:13}
\end{equation}
where $[\alpha] = a$ and $\gamma_i$ is defined in Definition~\ref{df:1}.
\label{rem:4}
\end{rem}
\begin{ex}
Let $a = [ d\phi ] \in H^1(\mathbb{S}^1; \mathbb{R})$, where $d\phi$ is the angular form on $\mathbb{S}^1$, then $a$ is a base of $H^1(\mathbb{S}^1; \mathbb{R})$.
Then index of the section $s$ at the point $(0, 0)$ in Example~\ref{ex:1} with respect to $a$ is equal to $n$.
\label{ex:3}
\end{ex}
\section{Connection and the Gauss-Bonnet theorem}
\label{sec:3}
Let $H$ be a connection in the fiber bundle $E$.
Let us denote by $V$ the \emph{vertical distribution}, that is the distribution of the tangent spaces to the fibers, then we have $TE = H \oplus V$.
\subsection{Vertical cohomology class}
\label{subsec:3_1}
\begin{stat}
Let $a \in H^1(F)$ and $H$ be a connection in $E$.
There exists a $1$-form $\alpha \in \Omega^1(E)$ such that
\begin{enumerate}
\item
$\left.\alpha\right|_H = 0$;
\item
for each $x \in M$, $d i_x^* \alpha = 0$ and $[ i_x^* \alpha ] = H^1(\eta_x) a$.
\end{enumerate}
\label{stat:1}
\end{stat}
\begin{proof}
Let us take an atlas $(U_i, \psi_i)$ of the bundle $\xi$, and fix a $1$-form $\bar\alpha \in \Omega^1(F)$ such that $[\bar\alpha] = a$.
Let $\rho_i$ be the partition of unity subordinate to the covering $\left\{ U_i \right\}$, and $\eta_i : \pi^{-1}(U_i) \to F$ (see~\eqref{eq:6}).
Then take the collection of $1$-forms $\bar\alpha_i = \eta^*_i \bar\alpha \in \Omega^1(\pi^{-1}(U_i))$, and then construct $\widetilde{\alpha} = \sum_i \pi^*\rho_i\, \bar\alpha_i \in \Omega^1(V)$.
Then
\begin{equation}
i_x^* \widetilde{\alpha} = \sum_i \rho_i(x)\, i_x^* \bar\alpha_i = \sum_i \rho_i(x)\, i_x^* \eta_i^* \bar\alpha \in \Omega^1(F_x).
\label{eq:14}
\end{equation}
Therefore,
\begin{equation}
d i_x^* \widetilde{\alpha} = \sum_i \rho_i(x)\, i_x^* \eta_i^* d\bar\alpha = 0
\label{eq:15}
\end{equation}
and
\begin{equation}
[ i_x^* \widetilde{\alpha} ] = \sum_i \rho_i(x)\, [ i_x^* \eta_i^* \bar\alpha ] = a,
\label{eq:16}
\end{equation}
because $\sum_i \rho_i(x) = 1$.
We have the decomposition $TE = H \oplus V$, and therefore we can define the desired $\alpha \in \Omega^1(E)$ by
$\left.\alpha\right|_H = 0$ and $\left.\alpha\right|_{V} = \widetilde{\alpha}$.
\end{proof}
\subsection{Linear connection in the vertical bundle $V \to E$}
\label{subsec:3_2}
Let $\Delta$ be a distribution on a manifold $B$.
A \emph{partial $\Delta$-connection} in a vector bundle $\pi : E \to B$ is given by a covariant derivative:
\begin{equation*}
\nabla : \Delta \times \Gamma(\pi)\to \Gamma(\pi)
\end{equation*}
such that
\begin{equation}
\begin{array}{l}
\nabla_{\lambda X + Y} s = \lambda \nabla_X s + \nabla_Y s,
\quad \lambda \in \mathcal{F}(B), X \in \mathfrak{X}(\Delta), s \in \Gamma(\pi);
\\
\nabla_{X} s_1 + s_2 = \nabla_X s_1 + \nabla_Y s_2;
X \in \mathfrak{X}(\Delta), s_1, s_2 \in \Gamma(\pi);
\\
\nabla_{X} f s = Xf \, s + f \nabla_X s;
\quad f \in \mathcal{F}(B), X \in \mathfrak{X}(\Delta), s \in \Gamma(\pi).
\end{array}
\label{eq:17}
\end{equation}
\subsubsection{Linear connection in the vector bundle $V \to E$ induced by the connection $H$}.
\label{subsubsec:3_2_1}
The connection $H$ induces a partial linear $H$-connection $\nabla^H$ in the vector bundle $V \to E$:
\begin{equation}
\nabla^H_X Y = \omega( [X, Y] ), \quad X \in \mathfrak{X}^h(E), Y \in \mathfrak{X}^v(E)
\label{eq:18}
\end{equation}
in the following way.
One can easily check that the properties~\eqref{eq:17} hold true.
The geometrical sense of the partial $H$-connection is as follows.
The connection $H$ induces a parallel translation of the fibers of the bundle $E$ along the paths on the base $M$: for $\gamma : [0, 1] \to M$, we have the diffeomorphism $\Pi_\gamma : F_{\gamma(0)} \to F_{\gamma(1)}$.
Therefore, if $\gamma^h$ is a horizontal lift of $\gamma$, we have the parallel translation $\Pi^H_{\gamma} = \left.d\Pi\right|_{\gamma^h(0)} : V_{\gamma^h(0)}E \to V_{\gamma^h(1)}E$ along the curve $\gamma^h$.
As $\Pi^H$ is linear, from this parallel translation we get a partial linear connection $\nabla^H$ which is exactly the connection~\eqref{eq:18}.

\subsubsection{Linear connection in the vector bundle $V \to E$}
\label{subsubsec:3_2_2}
Let us fix a linear partial $V$-connection $\nabla^V$ in the bundle $V \to E$, then the connection $\nabla^V$ induces a liner connection in each fiber.
We will assume that this connection does not have torsion.
Then let us consider the linear connection $\nabla = \nabla^H + \nabla^V$ in the bundle $V \to E$:
\begin{equation}
\nabla_X Y = \nabla^H_{X^h} Y + \nabla^V_{X^v} Y,
\label{eq:19}
\end{equation}
where $X = X^h + X^v$ corresponds to the decomposition $TE = H \oplus V$, and $Y \in \mathfrak{X}^v(E)$.

\subsection{Structure equations of connection}
\label{subsec:3_3}
\begin{stat}
Let $\xi = (\pi : E \to B)$ be a vector bundle of rank $r$ over a manifold $B$, and $\nabla$ be a linear connection in this bundle.
Then we can extend the exterior differential $d : \Omega^q(B) \to \Omega^{q + 1}(B)$ to a differential operator $D : \Omega^q(B) \otimes \Gamma(\xi) \to \Omega^{q + 1}(B) \otimes \Gamma(\xi)$, where $\Gamma(\xi)$ is the module of sections of $\xi$ so that
\begin{equation}
D(\omega \otimes s) = d\omega \otimes s + (-1)^{|\omega|} \omega \wedge \nabla s.
\label{eq:20}
\end{equation}
\label{stat:2}
\end{stat}
\begin{proof}
Let $\{e_\alpha\}_{\alpha = \overline{1, r}}$ be a local frame field of $\xi$ over $U \subset B$.
Then any $\omega \in \Omega^q(B) \otimes \Gamma(\xi)$ restricted to $U$ can be written in the form $\omega = \omega^\alpha e_\alpha$.
If we change the frame field to $\{e_{\alpha'}\}_{\alpha' = \overline{1, r}}$, then $e_{\alpha'} = A_{\alpha'}^\alpha e_\alpha$ and $\omega^{\alpha'} = A_{\alpha}^{\alpha'}\omega^\alpha$, where $A_{\alpha'}^\alpha$ are functions on $U$.
Then,
\begin{multline}
D(\omega^{\alpha'} e_{\alpha'}) = d\omega^{\alpha'} e_{\alpha'}+ (-1)^q \omega^{\alpha'} \wedge \nabla e_{\alpha'}
= d( A_\beta^{\alpha'} \omega^\beta ) A_{\alpha'}^\alpha e_\alpha + (-1)^q \omega^{\alpha'} \wedge \nabla (A_{\alpha'}^\alpha e_\alpha)
\\
= d A_\beta^{\alpha'} \wedge \omega^\beta A_{\alpha'}^\alpha e_\alpha+ A_\beta^{\alpha'} d\omega^\beta A_{\alpha'}^\alpha e_\alpha+ (-1)^q \omega^{\alpha'} \wedge dA_{\alpha'}^\alpha e_\alpha+ (-1)^q \omega^{\alpha'} \wedge A_{\alpha'}^\alpha \nabla e_\alpha
\\
= (A_{\alpha'}^\alpha d A_\beta^{\alpha'} \wedge \omega^\beta + (-1)^q \omega^{\beta} \wedge A_\beta^{\alpha'} dA_{\alpha'}^\alpha) e_\alpha+ d\omega^\alpha e_\alpha+ (-1)^q \omega^{\alpha} \wedge \nabla e_\alpha
\\
= (A_{\alpha'}^\alpha d A_\beta^{\alpha'} + A_\beta^{\alpha'} dA_{\alpha'}^\alpha) \wedge \omega^{\beta} e_\alpha + D(\omega^\alpha e_\alpha) = D(\omega^\alpha e_\alpha),
\label{eq:21}
\end{multline}
because $A_{\alpha'}^\alpha d A_\beta^{\alpha'} + A_\beta^{\alpha'} dA_{\alpha'}^\alpha = d( A_{\alpha'}^\alpha A_\beta^{\alpha'} ) = d \, \delta_\beta^\alpha = 0$.
\end{proof}
\begin{rem}
$D^2 (\omega \otimes s) = \omega \wedge R(s)$, where $R$ is the curvature tensor of the linear connection $\nabla$.
\label{rem:5}
\end{rem}
\begin{stat}
For $\omega \in \Omega^1(B) \otimes \Gamma(E)$, and $X, Y \in \mathfrak{X}(B)$,
\begin{equation}
D\omega(X, Y) = \nabla_X (\omega(Y)) - \nabla_Y (\omega(X)) - \omega([X, Y]).
\label{eq:22}
\end{equation}
\label{stat:3}
\end{stat}
\begin{proof}
Let $\left\{ e_\alpha \right\}$ be a local frame field of $\xi$, and $\omega = \omega^\alpha e_\alpha$.
Then
\begin{equation}
\nabla_X (\omega(Y) ) = \nabla_X (\omega^\alpha(Y) e_\alpha ) = X(\omega^\alpha(Y)) e_\alpha + \omega^\alpha(Y) \nabla_X e_\alpha,
\label{eq:23}
\end{equation}
therefore
\begin{multline}
\nabla_X (\omega(Y)) - \nabla_Y (\omega(X)) - \omega([X, Y]) =
\\
X(\omega^\alpha(Y)) e_\alpha + \omega^\alpha(Y) \nabla_X e_\alpha - Y(\omega^\alpha(X)) e_\alpha - \omega^\alpha(X) \nabla_Y e_\alpha - \omega^\alpha([X, Y]) e_\alpha =
\\
d\omega^\alpha(X, Y) e_\alpha - \omega^\alpha \wedge \nabla e_\alpha (X, Y) = D\omega (X, Y).
\label{eq:24}
\end{multline}
\end{proof}

Now let us assume that we have chosen the linear connection $\nabla$ in the vector bundle $V \to E$ given by~\eqref{eq:19}, and therefore we have the corresponding differential operator $D : \Omega^q(B) \otimes \mathfrak{X}^V(E) \to \Omega^{q + 1}(B) \otimes \mathfrak{X}^V(E)$, where $\mathfrak{X}^V(E)$ is the Lie algebra of vertical vector fields.
\begin{stat}[Structure equations]
Let $\omega : TE \to V$ be the connection form of the connection $H$.
Then
\begin{equation}
D\omega = \Omega,
\label{eq:25}
\end{equation}
where $\Omega(X, Y) = D\omega(X^h, Y^h) \in \Omega^2(E) \otimes \mathfrak{X}^V(E)$ is horizontal.
\label{stat:4}
\end{stat}
\begin{proof}
It is sufficient to prove that $D\omega(X, Y) = 0$ if $Y \in \mathfrak{X}^v(E)$.
We will apply~\eqref{eq:22}.
If $X \in \mathfrak{X}^h(E)$, then
\begin{equation}
D\omega(X, Y) = \nabla_X \omega(Y) - \nabla_Y \omega(X) - \omega( [X, Y] ) = \nabla_X Y - \omega( [X, Y] ) = 0
\label{eq:26}
\end{equation}
by the definition of $\nabla$ (see~\eqref{eq:18} and~\eqref{eq:19}).
If $X \in \mathfrak{X}^v(E)$, then
\begin{equation}
D\omega(X, Y) = \nabla_X \omega(Y) - \nabla_Y \omega(X) - \omega( [X, Y] ) = 0
\label{eq:27}
\end{equation}
again by the definition of $\nabla$ and because the connection $\nabla^V$ is torsion free.
\end{proof}
The tensor $\Omega \in \Omega^2(E) \otimes \mathfrak{X}^V(E)$ is called the \emph{curvature form of the connection $H$}.
\begin{cor}
The curvature form $\Omega$ vanishes if and only if $H$ is integrable.
\label{cor:1}
\end{cor}
\begin{proof}
Indeed, by Statement~\ref{stat:3}, we have that, for any horizontal vector fields $X$ and $Y$,
\begin{equation}
D\omega(X, Y) = - \omega([X, Y]),
\label{eq:28}
\end{equation}
therefore $D\omega(X^h, Y^h) = 0$, for any $X^h$ and $Y^h$ if and only if the Lie bracket of horizontal vector fields is again a horizontal vector field, this is equivalent to integrability of $H$.
\end{proof}

\subsection{Structure equations of the $1$-form $\alpha$}
\label{subsec:3_4}
The decomposition $TE = H \oplus V$ induces the bigrading on the space of differential forms on $E$: $\Omega^k(E) = \oplus_{l + m = k} \Omega^{(l, m)}(E)$, and $d : \Omega^{(0, l)}(E) \to \Omega^{(0, l+1)}(E)$ because the distribution $V$ is integrable.

Let $a \in H^1 F$ and $\widetilde{\alpha} \in \Omega^1(V)$ is a $1$-form such that
$\eta^*( [\left.\widetilde{\alpha}\right|_{F_x} ]) = a$.
In Statement~\ref{stat:1} we have constructed a $1$-form $\alpha \in \Omega^{(0, 1)}E)$ such that $d\alpha$ vanishes on each fiber.
Note that the form $\alpha(X) = \widetilde{\alpha}(\omega(X))$ satisfies the conditions (1) and (2) of Statement~\ref{stat:1}.

Let us calculate $d\alpha$.
As $X \widetilde{\alpha}(\omega(Y)) = \nabla_X \widetilde{\alpha}(\omega(Y)) + \widetilde{\alpha} (\nabla_X(\omega(Y))$, we have
\begin{multline}
d\alpha (X, Y) =X \alpha(Y) - Y \alpha(X) - \alpha( [X, Y] ) =
X \widetilde{\alpha}(\omega(Y)) - Y \widetilde{\alpha}(\omega(X)) - \widetilde{\alpha}(\omega([X, Y]) =
\\
\nabla_X \widetilde{\alpha}(\omega(Y)) + \widetilde{\alpha} (\nabla_X(\omega(Y))
- \nabla_Y \widetilde{\alpha}(\omega(X)) + \widetilde{\alpha} (\nabla_Y(\omega(X))
- \widetilde{\alpha}(\omega([X, Y]) =
\\
\nabla_X \widetilde{\alpha}(\omega(Y)) - \nabla_Y \widetilde{\alpha}(\omega(X)) +
\widetilde{\alpha}(D\omega).
\label{eq:29}
\end{multline}
Here we use Statement~\ref{stat:3}.
Let us consider the $2$-form
\begin{equation}
P \in \Omega^2(E), \quad P(X, Y) = \nabla_X \widetilde{\alpha}(\omega(Y)) - \nabla_Y \widetilde{\alpha}(\omega(X))
\label{eq:30}
\end{equation}
If $X, Y \in \mathfrak{X}^h(E)$, then $P(X, Y) = 0$.
Now if $X, Y \in \mathfrak{X}^v(E)$, then $P(X, Y) = d\widetilde{\alpha}(X, Y) = 0$.
Therefore, the form $P \in \Omega^{(1, 1)}(E)$, and for $X \in \mathfrak{X}^h(E)$ and $Y \in \mathfrak{X}^v(E)$, by~\eqref{eq:30}, we get that
\begin{equation}
P(X, Y) =\nabla_X \widetilde{\alpha} (\omega(Y)) =
X (\widetilde{\alpha}(Y)) - \widetilde{\alpha}(\nabla_X Y ) = X \alpha(Y) - \alpha( [X, Y] ) = L_X \alpha.
\label{eq:31}
\end{equation}
Now, using Statement~\ref{stat:3}, we arrive at the following statement.
\begin{stat}
The form $\alpha$ constructed in Statement~\ref{stat:1} lies in $\Omega^{(0, 1)}(E)$ and
$d \alpha = \theta_{(1, 1)} + \theta_{(2, 0)}$, where $\theta_{(1, 1)} \in \Omega^{(1, 1)}$ and $\theta_{(2, 0)} \in \Omega^{(2, 0)}$, and
\begin{equation}
\theta_{(1, 1)}(X, Y) = (L_X \alpha)(Y), \quad \theta_{(2, 0)} = \widetilde{\alpha}(\Omega).
\label{eq:32}
\end{equation}
\label{stat:5}
\end{stat}
Note that the forms $\theta_{(1, 1)}$ and $\theta_{(2, 0)}$ in~\eqref{eq:32} depends only on the connection $H$ and the form $\alpha$.
Now we apply the Stokes theorem to the surface $S$ and arrive at a generalization of the Gauss-Bonnet theorem:
\begin{thm}
For $a \in H^1(F)$, let $\alpha \in \Omega^1(E)$ be the corresponding vertical form constructed in Statement~\ref{stat:1}.
Also let $d\alpha = \theta_{(1, 1)} + \theta_{(2, 0)}$ as in Statement~\ref{stat:5}.
Then
\begin{equation}
\int\limits_S q^*\theta_{(1, 1)} + q^*\alpha(\Omega) = \sum\limits_{i = \overline{1, k}} index_{x_i}(s; a).
\label{eq:33}
\end{equation}
\label{thm:2}
\end{thm}
\begin{proof}
By the Stokes theorem,
\begin{equation}
\int_S d q^* \alpha = \int_{\partial S} q^*\alpha = \sum\limits_{i=1}^k \int_{\partial S_i} q^*\alpha.
\label{eq:34}
\end{equation}
We have that
\begin{equation}
\int_{\partial S_i} q^*\alpha = \int_{\mathbb{S}^1 \times \{ 1 \}} u_i^* q^* \alpha = \int_{\mathbb{S}^1} \gamma_i^* \alpha = \int_{\gamma_i} \alpha = ind_{x_i}(s; a),
\label{eq:35}
\end{equation}
Therefore,
\begin{equation}
\sum\limits_{i=1}^k \int_{\partial S_i} q^*\alpha = ind_{x_i}(s; a).
\label{eq:36}
\end{equation}
On the other hand, by Statement~\ref{stat:5},
\begin{equation}
\int_S d q^* \alpha = \int_S q^* d\alpha = \int_S q^*\theta_{(1, 1)} + q^*\theta_{(2, 0)}.
\label{eq:37}
\end{equation}
and thus we get the required statement.
\end{proof}

\section{Examples}
\label{sec:4}

\subsection{Classical Hopf theorem}
\label{subsec:4_1}
Let $(M, g)$ be a compact oriented two-dimensional Riemannian manifold, and $E = T^1(M, g)$ the bundle of unit vectors.
Then $E$ is a local trivial bundle with structure group $SO(2)$ and typical fiber $\mathbb{S}^1$.
Any $V \in \mathfrak{X}(M)$ defines a section $s = \frac{1}{\|V\|}V : M \setminus \Sigma \to E$, where $\Sigma$ is the set of zeros of $V$.
Assume that the set $\Sigma$ is discrete.
We will consider trivializations of the bundle $E$ given by local orthonormal frames $\left\{ e_1, e_2 \right\}$.
For each $x_i \in \Sigma$ let us consider a orientation preserving diffeomorphism $c_i : (D, 0) \to (U_i, x_i)$, where $D \subset \mathbb{R}^2$ is a disk centered at the origin $0$.
Let $\left\{ e_a \right\}$ be an orthonormal frame field over $U_i$ which determines a diffeomorphism $\pi^{-1}(U_i) \cong U_i \times \mathbb{S}^1$.
Now let consider a map
\begin{equation}
t_i : \mathbb{S}^1 \times [0, 1] \to U_i, \quad (\phi, \rho) \mapsto c_i( (1-\rho)\cos\phi, (1-\rho)\sin\phi) ),
\label{eq:38}
\end{equation}
and let $\psi_i : U_i\setminus\left\{ x_i \right\} \to \mathbb{S}^1$ be the local representation of the section $s$ with respect to the given trivialization, this means that
\begin{equation}
s = \cos\psi_i e_1 + \sin\psi_i e_2.
\label{eq:39}
\end{equation}
Now $t_i$ induces a diffeomorphism $\mathbb{S}^1 \times [0, 1) \to U_i \setminus \left\{ x_i \right\}$ and therefore we get a map $\psi_i \circ t_i^{-1} : \mathbb{S}^1 \times [0, 1) \to \mathbb{S}^1$.
Now assume that the map $\psi_i \circ t_i^{-1}$ can be prolonged to a map $\Psi_i: \mathbb{S}^1 \times [0, 1] \to \mathbb{S}^1$.
Under this assumption we can construct the desired compact $2$-dimensional manifold $S$ and the map $q : S \to E$ with properties (1) and (2) (see Section~\ref{sec:2}).
\begin{rem}
In geometric terms, our assumption means that the angle $\psi$ of the unit vector field $s$ with the vector field $e_1$ (see~\eqref{eq:39}) has limit when we approach to the point $x_i$ by some curve, but the limit depends on the curve.
\label{rem:6}
\end{rem}
In order to construct $S$, first we take $S' = M \setminus \cup_{i=1}^k U_i$, this is a submanifold with boundary which consists of the curves
$\Gamma_i = \partial U_i = c_i (\partial D)$, then
consider the disjoint union $S''$ of $k$ copies of the cylinder $\mathbb{S}^1 \times [0, 1] \cong \partial D \times [0, 1]$, and attach $S''$ to $S'$ by the map
$C : \sqcup_{i = 1}^k \partial D \times {0} \to \sqcup_{i = 1}^k \Gamma_i$ determined by the maps
\begin{equation}
\left. c_i \right|_{\partial D} : \mathbb{S}^1 \times \left\{ 0 \right\} \cong \partial D \times \left\{ 0 \right\} \to \Gamma_i.
\label{eq:40}
\end{equation}
Thus we get the manifold $S = S' \cup_C S''$.
Now the map $q : S \to E$ is defined as follows: for each $p \in S'$, we set $q(p) = s(p)$, and for each $p \in S'' = \sqcup_{i = 1}^k \partial D \times [0, 1]$, if $p$ belongs to $i$-th copy of $\mathbb{S}^1 \times [0, 1] \cong \partial D \times [0, 1]$, we set $q(p) = \Psi_i(p)$.
From the construction it is clear that $q$ is well defined on $S' \cap S'' = \sqcup_{i=1, k} \Gamma_i$.
Let $a \in H^1(\mathbb{S}^1)$ be the class represented by the angle form $\theta$.
The Riemannian connection $\nabla$ of the metric $g$ defines a connection $H$ in the bundle $E \to M$.
Let $\alpha \in \Omega^1(E)$ be the $1$-form constructed in Statement~\ref{stat:1}.
In terms of the trivialization
\begin{equation}
U \times \mathbb{S}^1 \to \pi^{-1}(U), \quad (q, \psi) \mapsto \cos\psi e_1(q) + \sin\psi e_2(q).
\label{eq:41}
\end{equation}
where $U \subset M$ and $\left\{ e_1, e_2 \right\}$ is an orthonormal frame field, the form $\alpha$ is written as follows:
\begin{equation}
\alpha = \theta + \pi^* G,
\label{eq:42}
\end{equation}
where $\nabla e_1 = G e_2$ and $\nabla e_2 = -G e_1$, and $G = G^2_{11} e^1 + G^2_{12} e^2$, $G^k_{ij}$ are the connection coefficients of $G$ with respect to the orthonormal frame $\left\{ e_1, e_2 \right\}$.
\begin{rem}
If $W = \cos\psi e_1 + \sin\psi e_2$, then $\nabla W = (d\psi + G)(-\sin\psi e_1 + \cos\psi e_2) = (d\psi + G) W^\perp$, therefore $W(t) = (\gamma(t), \psi(t))$ is parallel if and only if $\frac{d}{dt}\psi + G(\frac{d}{dt}\gamma) = 0$.
This means that, in terms of the trivialization, the connection form is exactly the form $\alpha$ given by~\eqref{eq:42}.
Therefore $\alpha$ vanishes over $H$ (condition (1) in Statement~\ref{stat:1}).
At the same $\alpha$ restricted to a fiber is $\theta$, this means satisfies condition (2) in Statement~\ref{stat:1}.
\label{rem:7}
\end{rem}
Denote by $\Psi^a_i : \mathbb{S}^1 \to \mathbb{S}^1$ the map $\Psi_i : \mathbb{S}^1 \times [0, 1] \to \mathbb{S}^1$ restricted to $\mathbb{S}^1 \times \left\{ a \right\}$.
Then $\Psi^0_i$ is homotopic to $\Psi^1_i$, therefore $\deg \Psi^0_i = \deg \Psi^1_i$.
However, $\Psi^0_i = \psi_i \circ t_i$ restricted to $\mathbb{S}^1 \times \left\{ 0 \right\}$ and its degree coincides with the degree of $\psi_i$ restricted to $\Gamma_i = \partial U_i$, therefore
the \emph{$\deg\Psi^0_i$ is equal to the classical index $Ind_{x_i}(s)$ of the vector field $s$ at the point $x_i$}.
Therefore,
\begin{equation}
ind_{x_i}(s, a) = \int_{\gamma_i} \alpha = \int_{\mathbb{S}^1} (\Psi^1_i)^* \theta = \deg \Psi^1_i = \deg \Psi^0_i = Ind_{x_i}(s).
\label{eq:43}
\end{equation}
This means that \emph{the classical index $Ind_{x_i}(s)$ of vector field $s$ at the singular point $x_i$ coincides with the index $ind_{x_i}(s, a)$}.
On the other hand $d\alpha = \pi^* dG = \pi^* K e^1 \wedge e^2$, where $K$ is the curvature of the metric $g$.
Therefore, in Statement~\ref{stat:5} $\theta_{(1, 1)} = 0$ and $\alpha(\Omega) = \pi^*(K e^1 \wedge e^2)$.
Thus, \eqref{eq:43} gives us
\begin{equation}
\int_S q^*\pi^*(K e^1 \wedge e^2) = \sum_{i = \overline{1, k}} Ind_{x_i}(s).
\label{eq:44}
\end{equation}
At the same time, by the construction, $\pi \circ q : S \setminus \partial S \to M \setminus \Sigma$ is a diffeomorphism, $\Sigma$ is a set of measure zero, and $K e^1 \wedge e^2$ is a $2$-form defined globally on $M$, therefore
\begin{equation}
\int_S q^*\pi^*(K e^1 \wedge e^2) = \int_M K e^1 \wedge e^2
\label{eq:45}
\end{equation}
and we get the classical result
\begin{equation}
\int_M K e^1 \wedge e^2 = \sum_{i = \overline{1, k}} Ind_{x_i}(s).
\label{eq:46}
\end{equation}

\subsection{Example: Projective bundles}
\label{subsec:4_2}
\subsubsection{Projective bundles}
\label{subsubsec:4_2_1}
\begin{df}
A \emph{projective bundle} is a locally trivial bundle $Q \to M$ with typical fiber $\mathbb{R}P^n$ and structure group $PGL(n)$.
The number $n$ is called the \emph{rank} of $Q$.
\label{df:3}
\end{df}
\begin{ex}
Let $E \to M$ be a vector bundle.
Then the bundle $P(E)$ of $1$-dimensional subspaces of the fibers of $E$ is called the \emph{projectivization} of $E$ and it is projective bundle.
\label{ex:4}
\end{ex}
\subsubsection{Projective bundle as projectivization of a vector bundle}
\label{subsubsec:4_2_2}
If we have a projective bundle $Q \to M$, then can we find a vector bundle $E$ such that $Q = P(E)$?
Note that for $A \in GL(n + 1)$, we have $\det (\lambda A) = \lambda^{n + 1} \det A$.
Now for an even $n$, the group $PGL(n) = GL(n + 1)/{A \sim \lambda A}$, $\lambda \in \mathbb{R}$, is isomorphic to $SL(n + 1)$ because for each equivalence class $[A] \in PGL(n)$ one can find a representative $B = (\det A)^{-1/(n + 1)} A \in [ A ]$ such that $\det B = 1$.
This $B$ is unique because for $B, B' \in [A]$ such that $\det B = \det B' = 1$, we have that $B' = \lambda B$ and $\lambda^{n + 1} = 1$, this implies that $\lambda = 1$.
Therefore in this case $PGL(n) \cong SL(n + 1)$.
For an odd $n$, the group $PGL(n) = PGL_{+}(n) \sqcup PGL_{-}(n)$, where
\begin{equation}
PGL_{+}(n) = \left\{ [ A ] \mid \det A > 0 \right\},
\quad
PGL_{-}(n) = \left\{ [ A ] \mid \det A < 0 \right\}.
\label{eq:47}
\end{equation}
By the same arguments as for case of even $n$ we get that $PGL_{+}(n) \cong SL(n)/{\pm I}$.
Note that $PGL_{-}(n)$ is diffeomorphic to $PGL_{+}(n)$.
If the structure group $PGL(n)$ of a projective bundle $Q$ reduces to $PGL_{+}(n)$, then $Q$ is called \emph{orientable}, on the contrary $Q$ is called \emph{no orientable}.
\begin{ex}
If a vector bundle $E$ is no orientable, then $P(E)$ is no orientable.
\label{ex:5}
\end{ex}
Now let consider the following situation.
Let $\widetilde{G}$ be a Lie group, $Z = Z(\widetilde{G})$ be the center of $\widetilde{G}$, and $G = \widetilde{G}/Z$.
So we have the exact sequence
\begin{equation}
\xymatrix{0 \ar[r] & Z \ar[r]^i &\widetilde{G} \ar[r]^p & G \ar[r] & 0.}
\label{eq:48}
\end{equation}
Let $P \to M$ be a $G$-principal bundle, $\mathcal{U} = \left\{ U_i \right\}$ be a good covering of a manifold $M$, and
$g_{ij} : U_{ij} \to M$ be the transition functions of $M$.
As the covering $\mathcal{U}$ is good, we can lift $g_{ij}$ up to $\widetilde{g}_{ij} : U_{ij} \to \widetilde{G}$ of $g_{ij}$, this means that the following diagram is commutative
\begin{equation}
\xymatrix{%
& \widetilde{G} \ar[d]_p
\\
U_{ij} \ar[ur]^{\widetilde{g}_{ij}} \ar[r]_{g_{ij}} & G%
}
\label{eq:49}
\end{equation}
Consider the Cech cochain $z = \left\{ z_{ijk} : U_{ijk} \to \widetilde{G} \right\}$ over the covering $\mathcal{U}$, where
\begin{equation}
z_{ijk} = \widetilde{g}_{ij} \widetilde{g}_{jk} \widetilde{g}_{ki}.
\label{eq:50}
\end{equation}
As $p(z_{ijk}) = g_{ij} g_{jk} g_{ki} = 1$, the functions $z_{ijk}$ take values in $Z$, therefore in fact $z \in \check{C}(\mathcal{U}; Z)$.
\begin{stat}
\begin{enumerate}
\item
The \v Cech cochain $z$ is a normalized cocycle.
\item
The \v Cech cohomology class $[ z ] \in \check{H}^2(M; Z)$ is trivial if and only if
there exists a $\widetilde{G}$-principal bundle $\widetilde{P}$ such that $P \cong \widetilde{P}/Z$.
\end{enumerate}
\label{stat:6}
\end{stat}
\begin{proof}
Let us prove that $\left\{ z_{ijk} \right\}$ is a \v Chech cocycle on $\mathcal{U}$.
First of all note that we have
\begin{equation}
\widetilde{g}_{ij} \widetilde{g}_{jk} = z_{ijk} \widetilde{g}_{ik} = \widetilde{g}_{ik} z_{ijk}.
\label{eq:51}
\end{equation}
Then
\begin{equation}
(\widetilde{g}_{ij} \widetilde{g}_{jk}) \widetilde{g}_{ki} = z_{ijk} \widetilde{g}_{ik} \widetilde{g}_{ki} = z_{ijk} \text{ and } \widetilde{g}_{ij} (\widetilde{g}_{jk} \widetilde{g}_{ki}) = \widetilde{g}_{ij} \widetilde{g}_{ji} z_{jki} = z_{jki},
\label{eq:52}
\end{equation}
therefore $z_{ijk} = z_{kij} = z_{jki}$.
At the same time,
\begin{equation}
z_{ijk}^{-1} \widetilde{g}_{ik}^{-1} = \left( z_{ijk} \widetilde{g}_{ik} \right)^{-1} = \left( \widetilde{g}_{ij} \widetilde{g}_{jk} \right)^{-1} = \widetilde{g}_{kj} \widetilde{g}_{ji} = z_{kji} \widetilde{g}_{ki},
\label{eq:53}
\end{equation}
hence follows that $z_{ijk}^{-1} = z_{kji} = z_{jik}$.
Moreover $z_{ijk} = 1$ if at least two indices coincide.
Thus $z = \left\{ z_{ijk} \right\}$ is a normalized cochain.
Now we have
\begin{equation}
(\widetilde{g}_{ij} \widetilde{g}_{jk}) (\widetilde{g}_{kj} \widetilde{g}_{jl}) = z_{ijk} \widetilde{g}_{ik} z_{kjl} \widetilde{g}_{kl} = z_{ijk} z_{kjl} z_{ikl} g_{il}
\label{eq:54}
\end{equation}
and
\begin{equation}
(\widetilde{g}_{ij} \widetilde{g}_{jk}) (\widetilde{g}_{kj} \widetilde{g}_{jl}) = (\widetilde{g}_{ij} (\widetilde{g}_{jk} (\widetilde{g}_{kj}) \widetilde{g}_{jl} = \widetilde{g}_{ij} \widetilde{g}_{jl} = z_{ijl} \widetilde{g}_{il},
\label{eq:55}
\end{equation}
thus $z_{ijk} z_{kjl} z_{ikl} = z_{ijl}$, this implies
\begin{equation}
\delta z = z_{jkl} z_{ikl}^{-1} z_{ijl} z_{ijk}^{-1} = 0.
\label{eq:56}
\end{equation}
If $z = \delta u$, where $u = \left\{ u_{ij} \in Z \right\}$, then
\begin{equation}
z_{ijk} = u_{ij} u_{jk} u_{ki}.
\label{eq:57}
\end{equation}
Therefore we can construct the $\widetilde{G}$-principal bundle $\widetilde{P} \to M$ with the transition functions $\widetilde{g}'_{ij} = \widetilde{g}_{ij} u_{ji}$ because
\begin{equation}
\widetilde{g}'_{ij} \widetilde{g}'_{jk} \widetilde{g}'_{ki} =\widetilde{g}_{ij} u_{ji} \widetilde{g}_{jk} u_{kj} \widetilde{g}_{ki} u_{ik} =\widetilde{g}_{ij} \widetilde{g}_{jk} \widetilde{g}_{ki} u_{ji} u_{kj} u_{ik} =
z_{ijk} u_{ji} u_{kj} u_{ik} = 1.
\label{eq:58}
\end{equation}
Now we have that $p(\widetilde{g}'_{ij}) = g_{ij}$ and from the following lemma it follows that there exists a principal bundle morphism $\pi : \widetilde{P} \to P$ written with respect to the trivializations as $(x, \widetilde{g}) \to (x, p(\widetilde{g}))$.
\begin{lm}
Let $G_a$, $a = 1, 2$ be Lie groups.
Let $\pi_a : P_a \to M$, $a = 1, 2$ be a $G_a$-principal bundle and $f : G_1 \to G_2$ be a Lie group homomorphism.
Assume that $\mathcal{U} = \left\{ U_i \right\}$ is a trivializing covering for $P_1$ and $P_2$ at the same time.
If the transition functions $g^a_{ij}$ with respect to the covering $\mathcal{U}$ are related by $f(g^1_{ij}) = g^2_{ij}$, then there exists a principal fiber bundle homomorphism $F : P_1 \to P_2$ corresponding to $f$, i.\, e. $F(p_1 g_1) = F(p_1) f(g_1)$.
\label{lm:1}
\end{lm}
\begin{proof}
We define $F_i : \pi_1^{-1}(U_i) \to \pi_2^{-1}(U_i)$ by the commutative diagram
\begin{equation}
\xymatrix{%
\pi_1^{-1}(U_i) \ar[r]^{F_i} \ar[d]_{\psi_i^1} & \pi_2^{-1}(U_i) \ar[d]^{\psi_i^2}
\\
U_i \times G_1 \ar[r]^{ 1 \times f} & U_i \times G_2%
}
\label{eq:59}
\end{equation}
where $\psi_a$ are the trivializations which are principal bundle isomorphisms, therefore $\psi_a(p_a g_a) = \psi_a(p_a) g_a$, where the action of $G_a$ on $U_i \times G_a$ is given by $(x, g_a) g'_a = (x, g_a g'_a)$.
From this follows that $F_i(p_1 g_1) = F_i(p_1) g_1$.
On the other hand, for $p_1 \in \pi^{-1}(U_{ij})$ we have $F_i(p_1) = F_j(p_1)$.
Indeed, we have the commutative diagram for $a = 1, 2$:
\begin{equation}
\xymatrix{%
U_{ij} \times G_a{}\ar[rr]^{(x, g_a) \to (x, g^a_{ji} g_a)} & & U_{ij} \times G_a
\\
&\pi_1^{-1}(U_{ij}) \ar[lu]^{\psi^a_i} \ar[ru]_{\psi^a_j} &%
}
\label{eq:60}
\end{equation}
Let $p_1 \in \pi^{-1}(U_{ij})$ and $\psi^1_i(p_1) = (x, g_1)$, then $\psi^1_j(p_1) = (x, g^1_{ji} g_1)$.
Also $\psi^2_i F_i(p_1) = (x, f(g_1))$, hence $\psi^2_j F_i(p_1) = (x, g^2_{ji} f(g_1))$.
Therefore,
\begin{equation}
\psi^2_j F_j(p_1) = (x, f(g^1_{ji} g_1)) = (x, f(g^1_{ji}) f(g_1))
= (x, g^2_{ji} f(g_1)) = \psi^2_j F_i(p_1),
\label{eq:61}
\end{equation}
thus $F_j = F_i$ on $\pi^{-1}(U_{ij})$ and therefore $\left\{ F_i \right\}$ determine a homomorphism $F : P_1 \to P_2$.
\end{proof}
From the construction of the bundle $\widetilde{P}$ it follows that $P = \widetilde{P}/Z$ and the principal fiber bundle homomorphism constructed in Lemma~\ref{lm:2} is exactly the projection $\widetilde{P} \to P=\widetilde{P}/Z$.
\end{proof}
Let $Q$ be a projective bundle of rank $n$, $n$ is odd.
Assume that $Q$ is orientable.
Let $\left\{ U_i \right\}$ be a good covering of $M$, $U_{i_1 \ldots i_k} = U_{i_1} \bigcap \cdots \bigcap U_{i_k}$, and $g_{ij} : U_{ij} \to PGL(n)$ the corresponding cocycle of transition functions.
As $U_{ij}$ are contractible, we can choose the maps $\hat{g}_{ij} : U_{ij} \to SL(n + 1)$ such that $p(\hat{g}_{ij}) = g_{ij}$, where $p : SL(n + 1) \to PGL_{+}(n)$ is the canonical projection.
Then $p(\hat{g}_{ij} \hat{g}_{jk} \hat{g}_{ki}) = g_{ij} g_{jk} g_{ki} = id \in PGL(n)$, therefore $e_{ijk} = \hat{g}_{ij} \hat{g}_{jk} \hat{g}_{ki}$ takes values in $\mathbb{Z}_2 = \left\{ \pm I \right\}$.
Note that $\mathbb{Z}_2 = \left\{ \pm I \right\}$ is the center of the group $SL(n + 1)$.
\begin{cor}
If $Q \to M$ is a projective bundle of odd rank $n$ and is oriented, then the obstruction to existence of a vector bundle $E$ such that $Q = P(E)$ is an element in $H^2(M, \mathbb{Z}_2)$.
\label{cor:2}
\end{cor}
\begin{df}
Let $Q \to M$ be a projective bundle of rank $n$.
The principal bundle $PGL(Q)$ with the group $PGL(n)$ adjoint to $Q$ is called the \emph{bundle of projective frames}.
\label{df:4}
\end{df}
\begin{ex}
On the total space of a vector bundle $E \to M$ the group $\mathbb{R}^{*} = \mathbb{R} \setminus \left\{ 0 \right\}$ acts freely by scaling, and $Q = P(E)$ is the quotient of $E$ with respect to this action.
Therefore, we have the $\mathbb{R}^{*}$-principal bundle $\pi : E \to Q$.
Also we have the induced action of $\mathbb{R}^{*}$ on the total space of the linear frame bundle $GL(E)$, and the quotient space is $PGL(Q)$.
The quotient map $GL(E) \to PGL(Q)$ is $\mathbb{R}^{*}$-principal bundle.
If $E$ is an orientable vector bundle of rank $n + 1$, where $n$ is odd, then the structure group $GL(n + 1)$ of $GL(E)$ reduces to $SL(n + 1)$, and in this case we have the double covering $SL(E) \to PGL(Q)$.
For an orientable vector bundle $E$ of rank $n + 1$, one can take a metric $g$ on $E$, therefore the $GL(n + 1)$-principal bundle $GL(E)$ of linear frames reduces to a $SO(n + 1)$-principal bundle of positively oriented orthonormal frames $SO(E)$.
Therefore, for $Q = P(E)$, the bundle $PGL(Q)$ reduces to the subgroup $SO(n + 1)/{\pm I} \subset PGL(n) = SL(n + 1)/{\pm I}$.
\label{ex:6}
\end{ex}

\subsubsection{Projective connections}
\label{subsubsec:4_2_3}
Let $Q \to M$ be a projective bundle.
A projective connection is a connection in $PGL(Q)$.
Any linear connection in $E$ induces a projective connection in $P(E)$.
Indeed, the parallel transport preserves one-dimensional subspaces in fibers of $E$.
At the same time the canonical projection $GL(E) \to PGL(P(E))$ is a principal bundle morphism corresponding to the morphism of groups $GL(n + 1) \to PGL(n)$.
Therefore, any connection in $GL(E)$ determines a unique connection in $PGL(E)$~\cite{kobayashi_nomizu1963}(Chapter II, Proposition 1).
\begin{stat}
Two linear connections $\nabla$ and $\nabla'$ in the vector bundle $E$ induce the same projective connection in $P(E)$ if and only if the deformation tensor $T = \nabla' - \nabla$ has the form
\begin{equation}
T_i{}^a_b = \xi_i \delta^a_b.
\label{eq:62}
\end{equation}
\label{stat:7}
\end{stat}
If $Q = P(E)$, where $E$ is an oriented vector bundle, then for any projective connection in $Q$ we can find an equi-affine connection $\nabla$ in $E$ with respect to a given volume form $\theta \in \Omega^{n + 1}E$, i.\,e. $\nabla \theta = 0$, such that $\nabla$ induces the given connection in $Q$.
Indeed in this case $SL(E, \theta)$ covers $PGL(Q) = SL(E, \theta)/\mathbb{Z}_2$, $\mathbb{Z}_2 = {\pm I}$, and any connection in $PGL(P(E))$ lifts to a connection in $SL(E, \theta)$.
\begin{stat}
Let $E$ be an oriented vector bundle, and $\theta \in \Omega^{n + 1}E$ be a volume form in $E$.
Then there is one-to-one correspondence between equi-affine connections on $E$ with respect to $\theta$ and the projective connections in $P(E)$.
\label{stat:8}
\end{stat}
\begin{rem}
In terms of parallel translation this correspondence can be explained as follows.
Let $\Gamma$ be a projective connection in $Q = P(E)$ induced by a linear connection $\nabla$ in $E$.
In terms of the parallel translation this means that a section $L$ of $Q$ along $\gamma$ is parallel with respect to the connection $\Gamma$ if there exists a section $s$ of $E$ along $\gamma$ which is parallel with respect to $\nabla$ and generates $L$.
This can be expressed also in terms of covariant derivative: $L = [ s ]$ is parallel with respect to $\Gamma$ if and only if $\nabla_{\dot\gamma} s = \mu s$.
Now any projective frame field $\mathcal{L} = \left\{ L_0, L_1, \cdots, L_n \right\}$ in $Q = P(E)$ determines a frame field $\mathcal{E} = \left\{E_1, \cdots, E_n\right\}$ in $E$ such that $L_i = [ E_i ]$, and $L_0 = [ E_0 ]$, where $E_0 = E_1 + E_2 + \cdots + E_n$, up to a scaling.
If in addition we have a volume form $\theta$, the projective frame field determines the linear frame field $\left\{ E_i \right\}$ up to a sign.
Then $\nabla_{\dot\gamma} E_i = \mu_i E_i$ and $\nabla_{\dot\gamma} E_0 = \mu_0 E_0$, hence follows that $\mu_0 = \mu_1 = \cdots = \mu_n$.
If we assume that the linear connection $\nabla$ is equi-affine with respect to a volume form $\theta$, then $\nabla \theta = 0$ and the function $\mu_0 = \frac{d}{dt} \log \theta(E_1, \cdots, E_n)$.
\label{rem:8}
\end{rem}
\subsection{Section of projective bundle and reduction of the bundle of projective frames}
\label{subsec:4_3}
\begin{stat}
Let $Q$ be a projective bundle of rank $1$.
Let $Aff(1) \subset PGL(1)$ be the subgroup of affine transformations of line.
Then $Q \cong PGL(Q)/Aff(1)$ and a section of projective bundle $Q$ is equivalent to the reduction of the bundle $PGL(Q)$ to the subgroup $Aff(1)$.
\label{stat:9}
\end{stat}
\begin{proof}
This follows from the fact that the group $PGL(1)$ acts transitively on $\mathbb{R}P^1$, and the isotropy subgroup of a point in $\mathbb{R}P^1$ is the group $Aff(1)$.
\end{proof}

\subsubsection{The structure equations of the form $\alpha$}
\label{subsubsec:4_3_1}
\subsubsection{The vertical form $\alpha$}
\label{subsubsec:4_3_2}
Let $E \to M$ be an oriented vector bundle and $rank E$ = 2.
Let $Q = P(E)$ be the corresponding projective bundle.
Also let us take a projective connection $\Gamma$ in $Q$.
The typical fiber of $Q$ is $\mathbb{R}P^1$.
Let us consider the canonical projection $j : \mathbb{R}^2\setminus \{0\} \to \mathbb{R}P^1$, $j(v^1, v^2) = [v^1 : v^2]$, and the corresponding double covering
$\bar \jmath : \mathbb{S}^1 \overset{\iota}{\hookrightarrow} \mathbb{R}^2\setminus \{0\} \overset{j}{\rightarrow} \mathbb{R}P^1$.
Let us consider the form
\begin{equation}
\theta = \frac{1}{2} \frac{1}{(v^1)^2 + (v^2)^2} \left(-v^2 dv^1 + v^1 dv^2\right) \in \Omega^1(\mathbb{R}^2 \setminus \left\{ 0 \right\}),
\label{eq:63}
\end{equation}
Then $\iota^*\theta$ is the angular form on $\mathbb{S}^1$ and there exists $\bar\alpha \in \Omega^1(\mathbb{R}P^1)$ such that
$j^*\bar\alpha = \theta$.
Then, as $\int_{\mathbb{S}^1} \iota^*\theta = 2\pi$, we have that $\int_{\mathbb{R}P^1} \bar\alpha = \pi$.
Therefore, $\bar\alpha$ is a generator of $H^1(\mathbb{R}P^1)$.
Now let us take a projective connection $H$ is $Q$ which is determined by an equi-affine connection $\hat H$ in $E$.
Now we can construct the form $\alpha$ in $\Omega^1(Q)$ such that $\left.\alpha\right|_H = 0$, $d i_x \alpha = 0$, and $i_x^*\alpha$ represents the class $[\bar \alpha] \in H^1(\mathbb{R}P^1)$.
Let us take a metric $g$ in $E$, the volume form $\theta$ associated with $g$, and let $\omega$ be the connection form on $E$, that is the projection onto the vertical subspace parallel to $\hat H$.
Now let us consider the $1$-form
\begin{equation}
\hat\alpha_v(X) = \frac{1}{2 g(v, v)}\theta(\omega(X), v) \in \Omega^1(E_0),
\label{eq:64}
\end{equation}
where $v \in E$, $E_0 = E \setminus 0(M)$, and $0$ is the zero section of $E$.
This form $\hat\alpha$ vanishes on $H$, and the metric $g$ allows us to choose the trivialization of the bundle $E$ such that the maps $i_x$ are isomorphisms of the euclidean vector spaces $(E_x, g_x)$ and $\mathbb{R}^n$ with the standard metric.
Therefore by $i_x$ the form $\hat\alpha$ restricted to the fibers is sent to the form~\eqref{eq:63}.
Now let us consider local coordinate system $(x^i, y^a)$ on $E$ such that $g = (y^1)^2 + (y^2)^2$, then $\theta = -y^2 dy^1 + y^1 dy^2$.
Consider the local coframe $dx^i$, $Dy^a = dy^a + \Gamma_i{ }^a_b y^b dx^i$.
\begin{lm}
\begin{equation}
d(Dy^a) = \Gamma_i{ }^a_b Dy^b \wedge dx^i + R_{ij}{ }^a_b y^b dx^i \wedge dx^j,
\label{eq:65}
\end{equation}
where $R_{ij}{ }^a_b = \partial_i \Gamma_j{ }^a_b - \partial_j \Gamma_i{ }^a_b + \Gamma_i{ }^a_c \Gamma_j{}^c_b - \Gamma_j{ }^a_c \Gamma_i{}^c_b$ is the curvature tensor of the linear connection in $E$.
\label{lm:2}
\end{lm}
\begin{proof}
We have
\begin{multline}
d(Dy^a) =d( dy^a + \Gamma_i{ }^a_b y^b dx^i ) =d\Gamma_i{ }^a_b y^b \wedge dx^i + \Gamma_i{ }^a_b dy^b \wedge dx^i =
\\
\partial_j\Gamma_i{ }^a_b y^b dx^j \wedge dx^i + \Gamma_i{ }^a_b (Dy^b - \Gamma_j{}^b_c y^c dx^j) \wedge dx^i =
\\
\Gamma_i{ }^a_b Dy^b \wedge dx^i +
\left(\partial_j\Gamma_i{ }^a_b y^b dx^j \wedge dx^i - \Gamma_i{ }^a_b \Gamma_j{}^b_c y^c dx^j \wedge dx^i \right) =
\\
\Gamma_i{ }^a_b Dy^b \wedge dx^i +
\left(\partial_j\Gamma_i{ }^a_b y^b dx^j \wedge dx^i + \Gamma_j{ }^a_b \Gamma_i{}^b_c y^c dx^j \wedge dx^i \right) =
\\
\Gamma_i{ }^a_b Dy^b \wedge dx^i + R_{ij}{ }^a_b y^b dx^i \wedge dx^j.
\label{eq:66}
\end{multline}
\end{proof}
The form $\hat{\alpha}$ is written with respect to this frame as follows:
\begin{equation}
\hat{\alpha} = \frac{1}{2}\frac{1}{(y^1)^2 + (y^2)^2} \left( -y^2 Dy^1 + y^1 Dy^2 \right).
\label{eq:67}
\end{equation}
Now let us assume that the connection is adapted to the metric $g$ in $E$, in this case
$\Gamma_i{}^a_b = -\Gamma_i{}^b_a$.
Then we have
\begin{equation}
d ((y^1)^2 + (y^2)^2 ) = 2 \left( y^1 Dy^1 + y^2 Dy^2 \right),
\label{eq:68}
\end{equation}
because
\begin{multline}
d ((y^1)^2 + (y^2)^2 ) =2 \left( y^1 dy^1 + y^2 dy^2 \right) =
\\
2 \left( y^1 (Dy^1 - \Gamma_i{}^1_2 y^2 dx^i) + y^2 (Dy^2 - \Gamma_i{}^2_1 y^1 dx^i) \right)
\\
= 2 \left( y^1 Dy^1 + y^2 Dy^2 \right).
\label{eq:69}
\end{multline}
Also we have that
\begin{equation}
\begin{array}{l}
dy^1 \wedge Dy^2 = Dy^1 \wedge Dy^2 - \Gamma_i{}^1_2 y^2 dx^i \wedge Dy^2,
\\[5pt]
y^1 d Dy^2 = \Gamma_i{}^2_1 y^1 Dy^1 \wedge dx^i + R_{ij}{ }^2_1 (y^1)^2 dx^i \wedge dx^j,
\\[5pt]
dy^2 \wedge Dy^1 = Dy^2 \wedge Dy^1 - \Gamma_i{}^2_1 y^1 dx^i \wedge Dy^1,
\\[5pt]
y^2 d Dy^1 = \Gamma_i{}^1_2 y^2 Dy^2 \wedge dx^i + R_{ij}{ }^1_2 (y^2)^2 dx^i \wedge dx^j,
\end{array}
\label{eq:70}
\end{equation}
Therefore,
\begin{multline}
d (y^1 Dy^2 - y^2 Dy^1) =
\\
\left( d y^1 \wedge Dy^2 + y^1 d D y^2 - dy^2 \wedge Dy^1 - y^2 d Dy^1 \right) =
\\
2 Dy^1 \wedge Dy^2 + R_{ij}{}^2_1 \left( (y^1)^2 + (y^2)^2 \right) dx^i \wedge dx^j.
\label{eq:71}
\end{multline}
Thus,
\begin{multline}
d \hat{\alpha} =
\\
\frac{1}{2} \left[ d\left(\frac{1}{(y^1)^2 + (y^2)^2}\right) \wedge \left( -y^2 Dy^1 + y^1 Dy^2 \right) + \frac{1}{(y^1)^2 + (y^2)^2} d\left( -y^2 Dy^1 + y^1 Dy^2 \right) \right] =
\\
- \frac{1}{((y^1)^2 +(y^2)^2)^2} \left( y^1 Dy^1 +y^2 Dy^2 \right) \wedge \left( -y^2 Dy^1 +y^1 Dy^2 \right) +
\\
+ \frac{1}{(y^1)^2 + (y^2)^2} Dy^1 \wedge Dy^2 + \frac{1}{2} R_{ij}{}^2_1 dx^i \wedge dx^j
\\
= \frac{1}{2} R_{ij}{}^2_1 dx^i \wedge dx^j.
\label{eq:72}
\end{multline}
Note that $\Omega = \pi^*(R_{ij}{}^2_1 dx^i \wedge dx^j)$ is the curvature form of the connection $H$ restricted to $\pi^{-1}(U)$, where $U$ is the coordinate neighborhood.
\begin{thm}
Let $E \to M$ be a vector bundle, $Q = P(E)$, and $j : P(E) \to Q$.
\begin{enumerate}
\item
There exist a $1$-form $\alpha \in \Omega^1(Q)$ such that $\hat{\alpha} = j^* \alpha$.
\item
$\left.\alpha\right|_H = 0$
\item
$d i_x \alpha = 0$, and $i_x^*\alpha$ represents the class $[\bar \alpha] \in H^1(\mathbb{R}P^1)$.
\item
$d\alpha = \frac{1}{2} \pi^*(\Omega)$.
\end{enumerate}
\label{thm:3}
\end{thm}
\begin{cor}
Let $E \to M$ be a vector bundle and $Q = P(E)$.
Assume that $g$ is a metric in $E$.
Let $\Omega$ be the curvature form of the connection in $Q$ induced by a connection in $E$ adapted to $g$.
Let $a \in H^1(\mathbb{R}P^1)$ is represented by the form $\bar \alpha$.
If $s : M \setminus \Sigma \to Q$ is a section with singularities,
\begin{equation}
ind (s; a) = \int_M \Omega.
\label{eq:73}
\end{equation}
\label{cor:3}
\end{cor}

\section{Gauss-Bonnet Theorem in Principal $G$-bundles with Singularities}
\label{sec:5}

\subsection{$G$-structures with singularities}
\label{subsec:5_1}
\begin{df}
Let $\bar{G}$ be a Lie group, and $G\subset \bar{G}$ be a Lie subgroup of $G$.
Let $M$ be a manifold, and $\Sigma$ be a submanifold of $M$.
A principal $\bar{G}$-bundle $\bar{P}(M, \bar{G})$ is called a \emph{principal $G$-bundle with singularities $\Sigma$} if the structure group $\bar{G}$ of $\bar{P}|_{M\setminus \Sigma}$ reduces to $G$.
\label{df:5}
\end{df}
\begin{ex}
Let $X$ be a vector field on a manifold $M$ with a discrete set of zeros $\Sigma \subset M$.
Then the frame bundle $L(M)$ is a principal $GL(n-1)$-bundle with singularities on $M$.
It is clear that on $M \setminus \Sigma$ the bundle $L(M)$ reduces to the principal $GL(n-1)$-subbundle consisting of frames with the first vector equal to the value of $X$ at the corresponding point.
\label{ex:7}
\end{ex}

\begin{ex}
If $M$ is a two-dimensional oriented Riemannian manifold, and $X$ is a vector field on $M$ with a discrete set of zeros $\Sigma \subset M$, then the orthonormal frame bundle $SO(M)$ is a principal $\{e\}$-bundle with singularities where the set of singularities $\Sigma$ because the vector field induced trivialization of $SO(M)$ over $M \setminus \Sigma$.
\label{ex:8}
\end{ex}
\begin{ex}
Let $M$ be a Riemannian manifold of dimension $n$. The principal $GL(n)$-bundle of frames $L(M)$ is a principal $O(n)$-bundle with an empty set of singularities.
\label{ex:9}
\end{ex}
Let us recall that there is an important relation between $H$-reductions of a principal $G$-bundle $P \to M$ and sections of the fiber bundle $P/H \to M$ which is given by the following statement.
\begin{stat}[\cite{kobayashi_nomizu1963}(Chapter I, Proposition 5.6)]
Let $P$ be a $G$-principal bundle and let $H$ be a Lie subgroup of $G$.
Then $G$ has a reduction to a principal $H$-bundle if and only if the associated bundle $P/H\rightarrow M$ admits a section.In particular, if $H=\{e\}$ then $P$ has a section.
\label{stat:10}
\end{stat}
\begin{cor}
Let $\bar{P}(M, \bar{G}, \pi)$ a principal $G$-bundle with singularities, and let $\Sigma$ be the set of singularities.
Then there is a section $s: M\setminus \Sigma \rightarrow \bar{P}/G$.
\label{cor:4}
\end{cor}

In what follows by $\mathfrak{\bar{g}}$ we denote the Lie algebra of $\bar{G}$, and by $\mathfrak{g}$ the Lie algebra of $G$.
Moreover, we suppose that the Lie group $\bar{G}$ is compact and connected.
Then $G$ is also compact and the quotient $\bar{G}/G$ is reductive, i.\,e, $\mathfrak{\bar{g}}$ has an $ad(G)$-invariant decomposition $\bar{\mathfrak{g}}=\mathfrak{g}\oplus \mathfrak{m}$.

Let $M$ be a $2$-dimensional compact connected oriented manifold.
Let $\bar{P}(M, \bar{G}, \bar{\pi})$ be a principal $G$-bundle with singularities, where the set of singularities $\Sigma$ consists of isolated points.
Consider the fiber bundle $\pi_E: E = \bar{P}/G \rightarrow M$.
Recall that the canonical projection $\pi_G: \bar{P}\rightarrow E=\bar{P}/G$ is a principal $G$-bundle~\cite{kobayashi_nomizu1963}(Chapter I, Proposition 5.5), and $\pi_G$ is a fiber bundle morphism: 
\begin{equation}
\xymatrix{%
\bar{P} \ar[rr]^{\pi_G} \ar[rd]_{\pi} & & E = \bar{P}/G \ar[ld]^{\pi_E}
\\
&M &%
}
\label{eq:74}
\end{equation}

Assume that $\bar{P}$ is endowed with a connection form $\bar{\omega}: T\bar{P} \rightarrow \bar{\mathfrak{g}}$, and let
$\bar{H} = \ker\bar{\omega}$ be the horizontal distribution of the connection, and $\bar{V}=\ker d\pi$ the vertical distribution on $\bar{P}$.
Note that $\bar{H}$ and $\bar{V}$ are vector subbundles of $T\bar{P}$, and $T\bar{P}=\bar{H} \oplus \bar{V}$.

Define the operator field $\widetilde{\omega} : T\bar{P} \to \bar{V}$ as follows:
\begin{equation}
\widetilde{\omega}_{\bar{p}}(\bar{X}) = \sigma_{\bar{p}}(\bar{\omega}_{\bar{p}}(\bar{X})),
\label{eq:75}
\end{equation}
where $\bar{X} \in T_{\bar{p}} \bar{P}$, and $\sigma_{\bar{p}}(\bar{a})$, $\bar{a} \in \bar{g}$, is the value of the fundamental vector field $\sigma(\bar{a})$ at the point $\bar{p}$.

\begin{lm}
a) $\widetilde{\omega}$ is a field of projectors onto the vertical distribution $\bar{V}$ with kernel $\bar{H}$.

b) $\widetilde{\omega}_{\bar{p}\bar{g}} \circ dR_{\bar{g}} (\bar{X}) = dR_{\bar{g}} \circ \widetilde{\omega}_{\bar{p}} (\bar{X})$, where $R_{\bar{g}} : \bar{G} \to \bar{G}$, $R_{\bar{g}} \bar{g}' = \bar{g}' \bar{g}$.

\label{lm:3}
\end{lm}
\begin{proof}
It is clear that $\widetilde{\omega}$ vanishes on $\bar{H}$.
For $\bar{X} \in \bar{V}_{\bar{p}}$, $\bar{p} \in \bar{P}$, one can find $\bar{a} \in \bar{\mathfrak{g}}$ such that $\sigma_{\bar{p}}(\bar{a}) = \bar{X}$.
Then, $\bar{\omega}_{\bar{p}}(\bar{X}) = \bar{a}$, hence follows that
$\widetilde{\omega}_{\bar{p}}(\bar{X}) = \sigma_{\bar{p}}(a) = X$.
Thus we get a).

Now a) implies b) because $dR_{\bar{g}} : T_{\bar{p}} \bar{P} \to T_{\bar{p}\bar{g}}\bar{P}$ is an isomorphism, which maps $\bar{H}_{\bar{p}}$ to $\bar{H}_{\bar{p}\bar{g}}$, and $\bar{V}_{\bar{p}}$ to $\bar{V}_{\bar{p}\bar{g}}$.
\end{proof}

\begin{stat}
a) There exists a operator field $\omega_E$ on $E$ such that, for each $\bar{p} \in \bar{P}$, the following diagram is commutative:
\begin{equation}
\xymatrix{
T_{\bar{p}} \bar{P} \ar[r]^{\widetilde{\omega}} \ar[d]_{d\pi_G} & \bar{V}_{\bar{p}} \ar[d]^{d\pi_G}
\\
T_{\pi_G(p)} E \ar[r]^{\omega_E} & VE_{\pi_G(p)}
}
\label{eq:76}
\end{equation}

b) $\omega_E$ is a projector onto $VE$ with kernel $HE = \ker\omega_E = d\pi_G(\bar{H})$.

c) The decomposition $T\bar{P}= \bar{H} \oplus \bar{V}$ of $T\bar{P}$ projects onto a decomposition $TE=HE\oplus VE$ of $TE$ via $d\pi_G$.
\label{stat:11}
\end{stat}
\begin{proof}
a) If $d\pi_G(\bar{X}_{\bar{p}}) = d\pi_G(\bar{X}'_{\bar{p'}})$, then $\bar{p}' = \bar{p} g$, and $\bar{X}' = dR_g \bar{X} + \bar{Y}$, where $g \in G$ and $\bar{Y} = \sigma_{\bar{p} g}(a) \in \ker d\pi_G$ with $a \in \mathfrak{g}$.
By Lemma~\ref{lm:3}, we have
\begin{equation}
\widetilde{\omega}_{\bar{p} g}(\bar{X}') =\widetilde{\omega}_{\bar{p} g}(dR_g \bar{X} + \sigma_{\bar{p} g}(a)) =
\widetilde{\omega}_{\bar{p} g}(dR_g \bar{X}) + \sigma_{\bar{p} g}(a) =
dR_g (\widetilde{\omega}_{\bar{p}}(X) + \sigma_{\bar{p} g}(a).
\label{eq:77}
\end{equation}
Then,
\begin{equation}
d\pi_G(\widetilde{\omega}_{\bar{p}}(\bar{X})) = d\pi_G(\widetilde{\omega}_{\bar{p}'}(\bar{X}')),
\label{eq:78}
\end{equation}
this proves that $\omega_E$ in the diagram~\eqref{eq:76} is well defined.

b) From the diagram~\eqref{eq:76} it follows that $d\pi_G(\bar{H}) \subset HE$.
For any $X \in HE_{\pi_G(\bar{p})}$, take $\bar{X} \in T_{\pi_G(\bar{p})} \bar{P}$ such that $d\pi_G(\bar{X}) = X$.
Then $d\pi_G(\widetilde{\omega}(\bar{X})) = 0$, this means that $\widetilde{\omega}(\bar{X}) = \sigma_{\bar{p}}(a)$, where $a \in \mathfrak{g}$.
Then $\widetilde{\omega}(\bar{X}) - \sigma_{\bar{p}}(a)) = 0$, therefore $\bar{X}) - \sigma_{\bar{p}}(a) \in \bar{H}_{\bar{p}}$ and $d\pi_G(\bar{X}) - \sigma_{\bar{p}}(a) \in \bar{H}_{\bar{p}}) = d\pi_G(\bar{X}) = X$.
Thus $HE \subset d\pi_G(\bar{H})$ and therefore $HE = d\pi_G(\bar{H})$.

For any $X \in VE_{\pi_G(\bar{p})}$, $X = d\pi_G(\bar{X})$, where $\bar{X} \in \bar{V}$.
Then, by Lemma~\ref{lm:3} and \eqref{eq:76}, we have that $\omega_E(X) = d\pi_G(\widetilde{\omega}(\bar{X})) = d\pi_G(\bar{X}) = X$.

c) follows from b).
\end{proof}
\begin{rem}
The bundle $E \to M$ is the bundle associated with the principal $\bar{G}$-bundle $\bar{P} \to M$ with the fiber $\bar{G}/G$ with respect to the left action of $\bar{G}$ on $\bar{G}/G$.
The distribution $HE$ is the connection induced on the associated bundle by the connection $\bar{H}$ on $\bar{P}$.
\label{rem:9}
\end{rem}

As $\bar{G}$ is compact, the homogeneous space $\bar{G}/G$ is reductive, this means that we have $ad(G)$-invariant decomposition $\bar{\mathfrak{g}} = \mathfrak{g} \oplus \mathfrak{m}$.
\begin{lm}
The vertical subbundle $VE$ is isomorphic to the bundle associated with the principal $G$-bundle $\pi_G : \bar{P} \to E$ with the fiber $\mathfrak{m}$ with respect to the left action of $G$ on $\mathfrak{m}$: $L_g (a) = ad(g^{-1}) a$.
\label{lm:4}
\end{lm}
\begin{proof}
The isomorphism $V_E \cong \bar{P}\times_{\bar{G}} \bar{G} \times \bar{\mathfrak{g}}/\mathfrak{g} \cong \bar{P}\times_G \mathfrak{m}$ follows from the fact that $VE=\bar{P}\times_{\bar{G}} T(\bar{G}/G)$~\cite{michor2000}(Chapter IV, section 18.18)
and that $T(\bar{G}/G)=\bar{G}\times \bar {\mathfrak{g}} / \mathfrak{g}$~\cite{sharpe2000}(Chapter 4, section 5).
In the exact form the isomorphism is written as follows:
\begin{equation}
\bar{P}\times_G \mathfrak{m} \to VE, \quad [\bar{p}, a] \mapsto d\pi_G(\sigma_{\bar{p}} (a)).
\label{eq:79}
\end{equation}
\end{proof}

\subsection{The Hopf-Gauss-Bonnet theorem for principal $G$-bundle with singularities}
\label{subsec:5_2}
Let $\bar{P} \to M$ be a principal $G \subset \bar{G}$-bundle with singularities, where $\bar{G}$ is a compact Lie group.
Let us take a connection $\bar{H}$ in $\bar{P}$ with the corresponding connection form $\omega$.
Let $E = \bar{P}/G$ and $H$ be the connection in $E$ induced by $\bar{H}$ with the connection form $\omega_E$ (see Statement~\ref{stat:11} and Remark~\ref{rem:9}). 

Given a $1$-cohomology class $a \in H^{1}(\bar{G}/G)$, by Statement~\ref{stat:1} we can construct a $1$-form $\alpha \in \Omega^{1}(E)$ on $E$ such that
\begin{enumerate}
\item
$\left.\alpha\right|_H = 0$;
\item
for each $x \in M$, $d i_x^* \alpha = 0$ and $[ i_x^* \alpha ] = H^1(\eta_x) a$.
\end{enumerate}

In this case one can construct the form $\alpha$ explicitly.
As $\bar{G}/G$ is a homogeneous space of a compact Lie group, we can take an invariant form $\xi \in \Omega^1_{inv}(\bar{G}/G) \cong \Lambda^1(\mathfrak{m})$ such that $[\xi] = a$ (see, e.\,g.,~\cite{dubrovin_novikov_fomenko1990}, Ch.1, 1).
Note that $\xi$ is determined by its value $\xi_0 : \mathfrak{m} \to \mathbb{R}$ at $[e] \in \bar{G}/G$.
Then from Lemma~\ref{lm:4} it follows that $\xi$ defines a form $\widetilde{\alpha} \in \Omega^1(VE)$:
\begin{equation}
\widetilde{\alpha}([\bar{p}, a]) = \xi_0(a).
\label{eq:80}
\end{equation}

\begin{thm}
Let $\omega_E: TE \rightarrow VE$ be a connection form induced by a connection in the principal $\bar{G}$-bundle $\bar{P} \to M$ (see~Statement~\ref{stat:11}). 

Then
\begin{equation}
\alpha(X)=\tilde{\alpha}(\omega_E(X))
\label{eq:81}
\end{equation}
The form $\alpha$ given by~\eqref{eq:81} satisfies the properties (1) and (2) of Statement~\ref{stat:1}.
\label{thm:4}
\end{thm}
\begin{proof}
For $X \in H$, $\omega_E(X) = 0$, therefore $\alpha(X) = 0$.

Recall that the trivializations of $E = P/G$ are constructed in the following way.
For a good covering $U_i$ of $M$ we take sections $s_i : U_i \to P$, then we define the charts
\begin{equation}
\psi_i : U_i \times \bar{G}/G \to \pi_E^{-1}(U_i), \quad (x, [\bar{g}]) \to [s_i(x) \bar{g}] = \pi_G(s_i(x) \bar{g}).
\label{eq:82}
\end{equation}
Therefore, for $x \in U_i$, $\eta_x^{-1} : \bar{G}/G \to E_x = \pi_E^{-1}(x)$, $\eta_x^{-1}[\bar{g}] = \pi_G(s_i(x)\bar{g})$.

Now let us fix $x \in U_i \subset M$ and let $\bar{p} = s_i(x)$.
For any $y \in E_x$ and $Y \in V_yE$, we have that $y = \pi_G(\bar{p} \bar{g}) = [\bar{p}, \bar{g}] = \eta_x^{-1}(\bar{g})$ and $Y = d\pi_G(\sigma_{\bar{p} \bar{g}}(a)) = [\bar{p} \bar{g}, a] = d\eta_x^{-1} (dL_{\bar{g}} a)$, where $a \in \mathfrak{m}$, and $L_{\bar{g}} : \bar{G}/G \to \bar{G}/G$, $L_{\bar{g}} \bar{g}' = \bar{g}\bar{g}'$.
Note that here we use the isomorphism constructed in Lemma~\ref{lm:4}. 
Then
\begin{equation}
\alpha(Y) = \widetilde{\alpha}(\omega_E(Y)) = \widetilde{\alpha}(Y) = \xi_0(a) = \xi(dL_{\bar{g}} a).
\label{eq:83}
\end{equation}
Therefore $\alpha(d\eta_x^{-1} (dL_{\bar{g}} a) ) = \xi(dL_{\bar{g}} a)$, thus $(\eta_x^{-1})^* \left.\alpha\right|_{E_x} = \xi$, then $\eta_x^* \xi = \left.\alpha\right|_{E_x}$.

Thus $d\left(\left.\alpha\right|_{E_x} \right) = 0$, and $\left[\left.\alpha\right|_{E_x} \right] = H^1(\eta_x) a$.
\end{proof}

The curvature form $\Omega_E= D\omega_E$ of the connection $\omega_E$ (see Section~\ref{sec:3}, equation~\eqref{eq:25}) can be represented by the following equality:
\begin{equation}
\Omega_E=1/2[\omega_E, \omega_E],
\label{eq:84}
\end{equation}
where $[ \ , \ ]: \Omega^{k}(E, TE)\times \Omega^{l}(E, TE)\rightarrow \Omega^{k+l}(E, TE)$ is the Frolicher-Nijenhuis bracket~\cite{michor2000}(Chapter IV, section 16.3).
In particular, for $1$-forms $\omega_1$ and $\omega_2$ this bracket is given by the following equation
\begin{align*}
[w_1, w_2](X, Y)&=[\omega_1(X), \omega_2(Y)]-[\omega_1(Y), \omega_2(X)]-\omega_2([\omega_1(X), Y])\-\\& \omega_1([\omega_2(X), Y])-[\omega_1(Y),X]-[\omega_2(Y), X]+(\omega_1 \circ \omega_2+\omega_2 \circ \omega_1)[X,Y].
\end{align*}

We repeat the arguments of subsection~\ref{subsec:3_4} and get that
\begin{equation}
d\alpha=L_X\alpha+1/2\tilde{\alpha}([\omega_E, \omega_E]).
\label{eq:85}
\end{equation}
Thus, by assuming the existence of the $2$-dimensional oriented compact manifold $S$ and the map $q:S\rightarrow E$ satisfying the properties given in Section~\ref{sec:1}, by Theorem~\ref{thm:2}, we obtain the following equality
\begin{equation}
\int_S(q^{*}L_X\alpha +\frac{1}{2}q^{*}\tilde{\alpha}([\omega_E, \omega_E]))=\sum_{i=\overline{1, k}} index_{x_i}(s,a),
\label{eq:86}
\end{equation}
where $s:M\setminus\Sigma \rightarrow E$ is the section corresponding to the reduction of $\bar{P}|_{M\setminus \Sigma}$ to $G$. The result given by the equation~\eqref{eq:86} will be called \emph{the Hopf-Gauss-Bonnet theorem for principal $G$-bundle with singularities}.

\begin{rem}
If $X$ is a vector field on $M$, and $\Sigma$ is the set of zeros of $X$, then $Y=\frac{X}{|X|}$ defines a section of the unit tangent bundle $M\rightarrow T^{1}M$ over $M\setminus \Sigma$.
Since $T^{1}M$ is the associated bundle to the bundle of orthonormal frames $M\rightarrow SO(M)$ with respect of standard action of $SO(2)$ on $\mathbb{S}^{1}$, then $SO(M)|_{M\setminus \Sigma}$ reduces to the trivial subgroup $G=\{e \}$ of $\bar{G}=SO(2)$.
Now, $\bar{G}/G=SO(2)$, and $SO(2)$ is diffeomorphic to $\mathbb{S}^{1}$, then $H^{1}(\bar{G}/G)$ is $1$-dimensional vector space generated by the angle form on $\mathbb{S}^{1}$.
Therefore, we can choose this generator to construct the form $\alpha\in \Omega^{1}(SO(M))$, and the connection form on $SO(M)$ induced by the Levi-Civita connection.
This connection form is given by equation~\eqref{eq:42}.
Thus, if $S$ and $q: S \rightarrow T^{1}M$ are the manifold and the map, respectively, constructed in section~\ref{sec:4}, the equation~\eqref{eq:86} reduces to the equation~\eqref{eq:46}.  
Hence, Hopf-Gauss-Bonnet theorem in principal $G$-bundle with singularities is a generalization of the classical Gauss-Bonnet theorem.
\label{rem:10}
\end{rem}

\subsection{Hopf-Gauss-Bonnet theorem for the case of the sum of Whitney of vector bundles}
\label{subsec:5_3}
Let $E_1, E_2, \cdots, E_k$ be $k$ vector bundles of rank $2$ over a compact connected Riemannian two-dimensional manifold $M$, and for each $i$, let $s^0_i: M \rightarrow E_i$ be the zero section of $E_i$.
Let $\tilde{E_i} = E_i \setminus s^0_i(M)$.

For all $i$ let us take sections $s_i : M \rightarrow E_i$.
Then $s_i$ is a section of $\tilde{E_i}$ with singularities, whose set of singularities is the set $\Sigma_i$ of zeros of $s_i$.
Thus, if we consider the product bundle $\tilde{E}=\tilde{E}_1\times \tilde{E}_2 \times \cdots \times \tilde{E}_k$ over $M\setminus \Sigma$, where $\Sigma=\cup \Sigma_i$, then the map $s=(s_1, \cdots,s_k) : M\setminus \Sigma \rightarrow \tilde{E}$ is a section with singularities of $\tilde{E}$.
In addition, suppose that the section $s : M\setminus \Sigma \rightarrow \tilde{E}$ admits a resolution of singularities (see section~\ref{sec:1}).
This means that there exists a $2$-dimensional oriented compact manifold $S$ with boundary $ \partial S =\bigcup_{j= \overline {1, k}} S_{j}$, and a map $q: S \rightarrow \tilde{E}$ such that $q$ restricted to $\overset{\circ}{S} = S \setminus \partial S$ is a diffeomorphism between $S \setminus \partial S$ and $s(M \setminus \Sigma)$.

Let us take connection forms $\omega_i: T\tilde{E}_i \rightarrow V\tilde{E}_i$ on $\tilde{E}_i$, for $i=1, \cdots, k$, then $\omega=(\omega_1, \cdots, \omega_k)$ is a connection form on $E$.
The curvature form of $\omega$ is $\Omega=(\Omega_1, \cdots, \Omega_k)$, where $\Omega_i$ is the curvature form of $\omega_i$.

Since the fiber of the bundle $\tilde{E}_i$ is $F_i=\mathbb{R}^{2}\setminus \{ 0 \}$, and the cohomology class $a_i=[d\phi_i]$ represented by the angular form on $F_i$ is a generator of $H^{1}(F_i)$, then the $1$-form $(d\phi_1, \cdots, d\phi_k)$ represents a nonzero cohomology class $a \in H^{1}(F)$, where $F=F_1\times \cdots \times F_k$ is the fiber of the bundle $E$.
Note that $F$ retracts by deformation to the $k$-torus $\mathbb{T}^k$.
With this in mind, we can choose the $1$-form $\alpha\in \Omega^{1}(E)$ satisfying properties (1) and (2) of Statement~\ref{stat:1} as follows:
\begin{equation}
\alpha = \omega_1 + \cdots + \omega_k,
\label{eq:87}
\end{equation}
where $\omega_i = d\varphi_i + \pi^* G_i$ is the connection form of $\tilde{E_i}$ (compare with~\eqref{eq:42}).

Then $d\alpha = \pi^* dG$, where $G = \sum_{i=1}^k G_i$, and by applying Theorem~\ref{thm:2} we obtain that
\begin{equation}
\int_S q^{*}\pi^{*}dG=\sum_{x_i \in \Sigma} index_{x_i}(s,a).
\label{eq:88}
\end{equation}

Let us fix an area form $\theta$ on $M$.
Let $dG_i = K_i \theta$, and call the function $K_i$ the \emph{curvature of the connection} $\omega_i$.

If we consider the points of $\Sigma \setminus \Sigma_i$ as singular points of the section $s_i$, then it is a section with singularities of the bundle $\tilde{E}_i \rightarrow M$.
In this sense we have the following statement.
\begin{thm}
Let $M$ be a compact connected two-dimensional manifold, let $E_i\rightarrow M, \ i=1, \cdots, k$, are $k$ vector bundles of rank $2$ over $M$, let $\tilde{E}_i =E_i\setminus s^0_i(M)$, where $s^0_i: M\rightarrow E_i$ is the zero section of $E_i$, and let $s_i : M \rightarrow E_i$ be a section of $E_i$ for each $i$.
If $\omega_1, \cdots, \omega_k$ are connections on $E_1, \cdots, E_k$ respectively, $K_1, \cdots, K_k$ are the curvatures of these connections, and $s = (s_1, \cdots, s_k) : M\setminus \Sigma \rightarrow \tilde{E}$, then
\begin{equation}
\sum_{j=\overline{1, k}}\int_M K_i \theta = \sum_{x_i \in \Sigma} \sum_{j=\overline{1, k}}ind_{x_i}(s_j, [d\phi_j]),
\label{eq:89}
\end{equation}
where $d\phi_i\in \Omega^{1}(\mathbb{R}^{2})$ is the angular form of $\mathbb{R}^{2}$, and $\theta$ is the area form on $M$.
\label{thm:5}
\end{thm}
\begin{proof}
Let $S$ be an oriented $2$-manifold with boundary, and let $q: S \rightarrow E$ a map which satisfies the condition given in section~\ref{sec:1}, that is, $q_i$ restricted to $\overset{\circ}{S} = S \setminus \partial S$ is a diffeomorphism between $S \setminus \partial S$ and $s_i(M \setminus \Sigma)$.
Now, because $E=\tilde{E}_1 \times \cdots \tilde{E}_k$, $q$ has the form $q=(q_1, \cdots, q_k)$, and for each $i$, the map $q_i: S\rightarrow \tilde{E}_i$ also satisfies this condition.
This fact follows from the commutative diagram
\begin{equation}
\xymatrix{
& M\setminus \Sigma \ar[r]^{id} \ar[d]_{s} & M\setminus \Sigma \ar[d]^{s_i}\\
\overset{\circ}{S}\ar[r]^{q} & s(M\setminus \Sigma)\ar[r]^{\tilde{p}_i} & s_i(M\setminus \Sigma)
}
\label{eq:90}
\end{equation}
Since the vertical arrows are diffeomorphisms, we have that $q_i$ is a diffeomorphism.
Therefore, for each $i = 1,\cdots, k$, the map $q_i=\tilde{p}_i\circ q$ is also a diffeomorphism.

Let us take the cohomology class $a=[(d\phi_1, \cdots, d\phi_k)]$ of the fiber of $\tilde{E}$, where $d\phi_i$ is the angular form on the fiber $F_i$ of $E_i$, then $a=[d\phi_1]+ \cdots +[d\phi_k]$.
Thus, the expression at the right side of equation~\eqref{eq:86} can be written as follows
\begin{equation}
\sum_{x_i \in \Sigma} index_{x_i}(s,a)=\sum_{x_i \in \Sigma} \sum_{j=\overline{1, k}}ind_{x_i}(s_j, [d\phi_j]).
\label{eq:91}
\end{equation}
Furthermore, $d\overset{i}{G}=K_i \theta$, and $d\alpha =\sum_{j=\overline{1, k}} \pi^{*}(K_i \theta)$.
Therefore, the expression at the left side of equation~\eqref{eq:86} reduces to the following equality
\begin{equation}
\int_S q^{*}\pi^{*}dG=\sum_{j=\overline{1, k}}\int_M K_i \theta
\label{eq:92}
\end{equation}
Hence, we obtain the following equation
\begin{equation}
\sum_{i=\overline{1, k}}\int_M K_i \theta = \sum_{x_i \in \Sigma} \sum_{j=\overline{1, k}}ind_{x_i}(s_j, [d\phi_j]).
\label{eq:93}
\end{equation}
\end{proof}

\begin{rem}
Note that if $\Sigma_i \cap \Sigma_j$ is non empty, and $x\in \Sigma_i \cap \Sigma_j$, then the indexes $ind_x(s_i)$ and $ind_x(s_j)$ could be different, still in the case where the vector bundles $E_i$ and $E_j$ are the same.
\label{rem:11}
\end{rem}


\begin{thebibliography}{10}

\bibitem{bott_tu1982}
Raoul Bott and Loring~W. Tu.
\newblock {\em Differential Forms in Algebraic Topology}.
\newblock Springer New York, 1982.

\bibitem{kobayashi_nomizu1969}
S.~Kobayashi and Nomizu K.
\newblock {\em Foundations of Differential Geometry. Vol. II}.
\newblock Wiley, New York, London, 1969.

\bibitem{chern_1944}
Shiing-Shen Chern.
\newblock A simple intrinsic proof of the gauss-bonnet formula for closed
  riemannian manifolds.
\newblock {\em The Annals of Mathematics}, 45(4):747--752, oct 1944.

\bibitem{arteaga_malakhaltsev_trejos_2012}
Jos{\'{e}} Ricardo~Arteaga B, Mikhail Malakhaltsev, and Alexander Haimer~Trejos
  Serna.
\newblock Isometry group and geodesics of the wagner lift of a riemannian
  metric on two-dimensional manifold.
\newblock {\em Lobachevskii J Math}, 33(4):293--311, oct 2012.

\bibitem{agafonov2009}
S.~I. Agafonov.
\newblock On implicit odes with hexagonal web of solutions.
\newblock {\em J. Geom. Anal.}, 19:481--508, 2009.

\bibitem{arias_arteaga_malakhaltsev2015}
Arias~F. A., J.~R. Arteaga, and M.~A. Malakhaltsev.
\newblock 3-webs with singularities.
\newblock {\em Lobachevskii J. of Math}, 2015.
\newblock to appear.

\bibitem{arteaga_malakhaltsev2011}
J.~R. Arteaga and M.~A. Malakhaltsev.
\newblock Symmetries of sub riemannian surfaces.
\newblock {\em J. Geom. Phys.}, 61:290--308, 2011.

\bibitem{kobayashi_nomizu1963}
S.~Kobayashi and Nomizu K.
\newblock {\em Foundations of Differential Geometry. Vol. I}.
\newblock Wiley, New York, London, 1963.

\bibitem{michor2000}
Michor.
\newblock {\em Topics in differential geometry}.
\newblock Springer-Verlag, 2000.

\bibitem{sharpe2000}
Sharpe~R. W.
\newblock {\em Differential geometry}.
\newblock Springer, 2000.

\bibitem{dubrovin_novikov_fomenko1990}
B.~A. Dubrovin, A.~T. Fomenko, and S.~P. Novikov.
\newblock {\em Modern Geometry - Methods and Applications Part III.
  Introduction to Homology Theory}.
\newblock Springer-Verlag, 1990.

\end{thebibliography}
\end{document}